\documentclass[12pt, final]{l4dc2022}

\usepackage{multirow}
\usepackage{float}
\usepackage{enumitem}
\usepackage{booktabs}
\usepackage{cancel}
\usepackage{algorithmic}
\usepackage{algorithm}
\usepackage{wrapfig}

\usepackage{color,hyperref, colortbl}

\newtheorem{assumption}{Assumption}

\definecolor{light-gray}{gray}{0.95}
\definecolor{LightCyan}{rgb}{0.88,1,1}

\newcommand{\ex}[1]{\mathbb{E} \left[ #1 \right]}

\title[Convergence and Stability of the Stochastic Proximal Point Algorithm with Momentum]{Convergence and Stability of the \\ Stochastic Proximal Point Algorithm with Momentum}
\usepackage{times}

\author{%
 \Name{Junhyung Lyle Kim} \Email{jlylekim@rice.edu}\\
\addr {Department of Computer Science, Rice University}
 \AND
 \Name{Panos Toulis} \Email{panos.toulis@chicagobooth.edu}\\
\addr {Booth School of Business, University of Chicago}
 \AND
 \Name{Anastasios Kyrillidis} \Email{anastasios@rice.edu}\\
\addr {Department of Computer Science, Rice University}
}

\begin{document}

\maketitle

\begin{abstract}%
 Stochastic gradient descent with momentum (SGDM) is the dominant algorithm in many optimization scenarios, including convex optimization instances and non-convex neural network training. Yet, in the stochastic setting, momentum interferes with gradient noise, often leading to specific step size and momentum choices in order to guarantee convergence, set aside acceleration. Proximal point methods, on the other hand, have gained much attention due to their numerical stability and elasticity against imperfect tuning. Their stochastic accelerated variants though have received limited attention: how momentum interacts with the stability of (stochastic) proximal point methods remains largely unstudied. To address this, we focus on the convergence and stability of the stochastic proximal point algorithm with momentum (SPPAM), and show that SPPAM allows a faster linear convergence to a neighborhood compared to the stochastic proximal point algorithm (SPPA) with a better contraction factor, under proper hyperparameter tuning. In terms of stability, we show that SPPAM depends on problem constants more favorably than SGDM, allowing a wider range of step size and momentum that lead to convergence.
\end{abstract}


\section{Introduction}
\label{sec:intro}
\paragraph{Background.}
We focus on unconstrained empirical risk minimization instances \citep{robbins_stochastic_1951, polyak1992acceleration, bottou_stochastic_2012, bottou2007tradeoffs, shalev2011pegasos, nemirovski_robust_2009, moulines_non-asymptotic_2011, bach2013non}, as in: 
\vspace{-0.2cm}
\begin{align}
\label{eq:obj}
    \min_{x \in \mathbb{R}^p} ~f(x) = \frac{1}{n} \sum_{i=1}^n f_i(x). 
    \\[-20pt] \nonumber
\end{align}

To solve \eqref{eq:obj}, stochastic gradient descent (SGD) has become the de facto method used by the machine learning community,
mainly due to its computational efficiency \citep{zhang_2004_solving, bottou_stochastic_2012, bottou_optimization_2018}. 
For completeness, SGD iterates as follows: 
\vspace{-0.1cm}
\begin{align}
\label{eq:sgd}
    x_{t+1} = x_t  - \eta \nabla f_{i_t}(x_t), 
    \\[-18pt] \nonumber
\end{align}
where $\eta$ is the step size, and $\nabla f_i$ is the (stochastic) gradient computed at the $i$-th data point.

\paragraph{Properties of SGD and Its Momentum Extension.}
While computationally efficient, stochastic methods often suffer from two major limitations: $(i)$ slow convergence, and $(ii)$ numerical instability.
Due to the gradient noise, SGD could take longer to converge in terms of iterations \citep{moulines_non-asymptotic_2011, gower_sgd_2019}. 
Moreover, SGD suffers from numerical instabilities both in theory \citep{nemirovski_robust_2009} and practice \citep{bottou_stochastic_2012}, allowing only a small range of $\eta$ values that lead to convergence, which often depend on unknown quantities \citep{moulines_non-asymptotic_2011}.

With respect to slow convergence, many variants of accelerated methods have been proposed, along with analyses \citep{su2014differential, defazio2019curved, laborde2019lyapunov, allen2017linear, lessard2016analysis, hu2017dissipativity, wibisono2016variational, bubeck2015geometric}.
Most notable cases include the Polyak's momentum \citep{polyak_methods_1964, polyak1987introduction} and the Nesterov's acceleration methods \citep{nesterov_lectures_2018, ahn_proximal_2020, nesterov1983method}. 
These methods allow faster (sometimes optimal) 
convergence rates, while having virtually the same computational cost as SGD. 
In particular, SGD with momentum (SGDM) \citep{polyak_methods_1964, polyak1987introduction} iterates as follows: 
\vspace{-0.6cm}
\begin{align}
\label{eq:sgdm}
    x_{t+1} = x_t  - \eta \nabla f_{i_t}(x_t) + \beta (x_t - x_{t-1}),
    \\[-20pt] \nonumber
\end{align}
where $\beta \in [0,1)$ is the momentum parameter.
The intuition is that, if the direction from $x_{t-1}$ to $x_t$ was ``correct," then SGDM utilizes this inertia weighted by the momentum parameter $\beta,$ instead of just relying on the current point $x_t.$ 
Much of the state-of-the-art performance has been achieved with SGDM \citep{huang2017densely, howard2017mobilenets, he2016deep}.

Yet, SGDM could be hard to tune: SGDM adds another hyperparameter---momentum $\beta$---to an already sensitive stochastic procedure of SGD.
As such, various works have found that such motions could aggravate the instability of SGD. 
For instance, \cite{liu_accelerating_2019} and \cite{kidambi_insufficiency_2018} show that accelerated SGD does not in general provide any acceleration over SGD, regardless of careful tuning; further, accelerated SGD may diverge for the step sizes that SGD converges. 
\cite{assran_convergence_2020} also show that, even with finite-sum of quadratic functions, accelerated SGD may diverge under the usual choices of step size and momentum.
See also \cite{loizou2020momentum, devolder2014first, d2008smooth} for more discussions on this topic.

\paragraph{Stability via Proximal Updates.}
With respect to numerical stability, variants of SGD that utilize proximal updates have recently been proposed \citep{ryu_stochastic_2017, toulis_statistical_2014, toulis_asymptotic_2017, toulis_proximal_2021, asi_stochastic_2019, asi_minibatch_2020}.
In particular, \cite{toulis_proximal_2021} introduced stochastic errors in proximal point algorithms (SPPA) and analyzed its convergence and stability, which iterates similar to: \vspace{-0.2cm}
\begin{align}
    x_{t+1} &= x_t - \eta \left( \nabla f(x_{t+1}) + \varepsilon_{t+1} \right)   \label{eq:stoc-ppa}. 
    \\[-21pt] \nonumber 
\end{align}
Without stochastic errors, \eqref{eq:stoc-ppa} is known as the proximal point algorithm (PPA) \citep{rockafellar_monotone_1976, guler_convergence_1991} or the implicit gradient descent (IGD). 
PPA/IGD is known to converge with minimal assumptions on hyperparameter tuning, by improving the conditioning of the optimization problem; more details in Section~\ref{sec:prelim}.
In the stochastic setting, \cite{toulis_proximal_2021} show that
SPPA enjoys an exponential discount of the initial condition, regardless of the step size $\eta$ and the smoothness parameter $L$. 
On the contrary, for SGD, 
both $\eta$ and $L$ show up within an exponential term, amplifying the initial conditions, leading to even divergence if misspecified \citep{moulines_non-asymptotic_2011}.

\paragraph{Our Focus and Contributions.}
Stochastic accelerated variants of PPA have received limited attention: how momentum interacts with the stability that PPA provides remains unstudied. 
To the best of our knowledge, \textit{no momentum has been considered for stochastic proximal point updates that, beyond convergence, also studies the stability of the acceleration motions.}
This is the aim of this work. 
Our contributions are summarized as:
\vspace{-0.1cm}
\begin{itemize}[leftmargin=*]
    \item We introduce stochastic PPA with momentum (SPPAM), and study its convergence and stability behavior. SPPAM directly incorporates the momentum term akin to \eqref{eq:sgdm} into \eqref{eq:stoc-ppa}: \vspace{-0.1cm}
\begin{align}
    x_{t+1} &= x_t - \eta \left (\nabla f(x_{t+1}) + \varepsilon_{t+1}\right) + \beta (x_t - x_{t-1})  \label{eq:acc-stoc-ppa}. 
    \\[-21pt] \nonumber
\end{align}
We study whether adding momentum $\beta$ results in faster convergence akin to SGDM, while preserving the numerical stability, inherited by utilizing proximal updates akin to SPPA. 
\vspace{-0.15cm}
\item We show that SPPAM enjoys linear convergence to a neighborhood (Theorem \ref{thm:lin-conv}) with a better contraction factor than SPPA (Lemma~\ref{lem:SPPAM-contraction}). We further characterize the conditions on $\eta$ and $\beta$ that result in acceleration (Corollay~\ref{cor:acc-condition}). 
Finally, we characterize the condition that leads to the exponential discount of initial conditions for SPPAM (Theorem~\ref{thm:init-discount-condition}), which is significantly easier to satisfy compared to SGDM.
\vspace{-0.15cm}
\item Empirically, we confirm our theory with experiments on generalized linear models (GLM), including
linear and Poisson regressions with different condition numbers. 
As expected, 
SGD and SGDM converge only for specific choices of $\eta$ and $\beta$, while SPPA converges for a much wider range of $\eta.$ 
SPPAM enjoys the advantages of both acceleration from the momentum and stability from the proximal step: it converges for the range of $\eta$ that SPPA converges but with faster rate, which improves or matches that of SGDM, when the latter converges.
\end{itemize}

\section{Preliminaries} 
\label{sec:prelim}

\renewcommand{\arraystretch}{1.0}
\begin{table*}[ht]
    \centering
    \caption{Comparison of different algorithms in Section \ref{sec:prelim}. $(\cdot)^{\alpha}$ is {{\cite{rockafellar_monotone_1976, guler_convergence_1991}}}; $(\cdot)^{\textcolor{red}{\beta}}$ is {{\cite{guler_new_1992, lin_universal_2015, lin_catalyst_2018}}}; 
    $(\cdot)^{\textcolor{blue}{\gamma}}$ is {{\cite{toulis_statistical_2014, toulis_asymptotic_2017, ryu_stochastic_2017}}}; 
    $(\cdot)^{\textcolor{magenta}{\delta}}$ is {{\cite{asi_stochastic_2019, asi_minibatch_2020}}}; 
    $(\cdot)^{\textcolor{purple}{\epsilon}}$ is {{\cite{kulunchakov_generic_2019}}};
    $(\cdot)^{\textcolor{orange}{\zeta}}$ is {{\cite{chadha_accelerated_2021}}};
    $(\cdot)^{\textcolor{olive}{\eta}}$ is {{\cite{deng_minibatch_2021}}}. 
    We highlight with color the algorithms that include momentum motions.}
    \label{tab:algos}
    \resizebox{\textwidth}{!}{
    \begin{tabular}{c c c c c}
    \toprule
        \rowcolor{light-gray}
        Method & & & & Deterministic  \\
        \cmidrule{1-1} \cmidrule{5-5}
        \multirow{2}{*}{PPA/IGD$^{\alpha}$} & 
        & & & $x_{t+1} = \arg \min_x \left\{ f(x) + \tfrac{1}{2\eta_t} \|x-x_t\|_2^2 \right\}$  \\
        & & & & $\Leftrightarrow x_{t+1} = x_t - \eta_t \nabla f(x_{t+1})$ \\ 
        \cmidrule{1-1} \cmidrule{5-5}
        \multirow{3}{*}{Acc. PPA/Catalyst$^{\textcolor{red}{\beta}}$} & & & & 
         $x_{t+1} \approx    \arg\min_x \left\{ f(x) + \tfrac{\kappa}{2} \|x-y_t \|_2^2 \right\}$  \\
        & & & & $y_t = x_t + \textcolor{red}{\beta_t (x_t - x_{t-1})}$  \\
        & & & & where $\alpha_t^2 = (1-\alpha_t) \alpha_{t-1}^2 + \tfrac{\mu}{\mu+\kappa}\alpha_t, \quad \beta_t = \tfrac{\alpha_{t-1}(1-\alpha_{t-1})}{\alpha_{t-1}^2 +\alpha_t}$ \\
        \bottomrule
        \rowcolor{light-gray} & & & & Stochastic  \\
        \cmidrule{1-1} \cmidrule{5-5}
        \multirow{2}{*}{SPI/ISGD$^{\textcolor{blue}{\gamma}}$}
        & & & & $ x_{t+1} = \arg \min_x \left\{ f_{i_t} (x) + \tfrac{1}{2\eta_t} \|x-x_t\|_2^2 \right\}  $ \\
        & & & & $\Leftrightarrow x_{t+1} = x_t - \eta_t \nabla f_{i_t}(x_{t+1})$ \\
        \cmidrule{1-1} \cmidrule{5-5}
        APROX$^{\textcolor{magenta}{\delta}}$ & & & & Set $f_{i_t}(x) := \max \left\{ f_{i_t}(x_t) + \langle \nabla f_{i_t}(x_t), x-x_t \rangle, \inf_z f_{i_t}(z) \right\}$ from SPI \\
        \cmidrule{1-1} \cmidrule{5-5}
        \multirow{1}{*}{Stochastic Catalyst$^{\textcolor{purple}{\epsilon}}$} & & & & 
        Set $f(x) := f(y_t) + \langle g_t, x-y_{t} \rangle + \tfrac{\kappa + \mu}{2} \|x-y_t\|_2^2$ from Catalyst \\
        \cmidrule{1-1} \cmidrule{5-5}
        \multirow{4}{*}{Acc. APROX$^{\textcolor{orange}{\zeta}}$} 
        & & & &  $y_t = \textcolor{red}{(1-\beta_t) x_t + \beta_t z_t}$  \\
        & & & & $z_t = \arg \min_x \left\{ f_{i_t}(x) + \frac{1}{\eta_t} \|x-z_t\|_2^2 \right\}$ \\
        & & & &  $x_{t+1} = \textcolor{red}{(1-\beta_t) x_t + \beta_t z_{t+1}}$\\
        & & & &  where $f_{i_t}(x) := \max \left\{ f_{i_t}(x) + \langle \nabla f_{i_t}(x), y-x \rangle, \inf_z f_{i_t}(z) \right\}$  \\
        \cmidrule{1-1} \cmidrule{5-5}
        \multirow{2}{*}{SEMOD$^{\textcolor{olive}{\eta}}$}
        & & & & $ y_t = x_t + \textcolor{red}{\beta (x_t - x_{t-1})} $ \\
        & & & & $ x_{t+1} = \arg \min_x \left\{ f_{i_t} (x) + \tfrac{1}{2\eta_t} \|x-y_t\|_2^2 \right\}  $ \\
        \cmidrule{1-1} \cmidrule{5-5}
        \rowcolor{LightCyan} \multirow{1}{*}{SPPAM (this work)} & & & 
        & $x_{t+1} = x_t - \eta \left( \nabla f(x_{t+1}) + \varepsilon_{t+1} \right) + \textcolor{red}{\beta (x_t - x_{t-1})}$ \\
        \bottomrule
    \end{tabular}
    }
\end{table*}

\paragraph{Proximal Point Algorithm (PPA).}
The proximal point algorithm (PPA) \citep{rockafellar_monotone_1976, guler_convergence_1991} obtains the next iterate for minimizing $f$ by solving the following optimization problem: 
\vspace{-0.1cm}
\begin{align}
\label{eq:ppa}
    x_{t+1} = \arg \min_{x\in \mathbb{R}^p} \left\{ f(x) + \tfrac{1}{2\eta} \| x - x_t \|_2^2 \right\}, \\[-20pt] \nonumber
\end{align}
which is equivalent to implicit gradient descent (IGD) by the first-order optimality condition: 
\vspace{-0.1cm}
\begin{align}
\label{eq:igd}
    x_{t+1} = x_t - \eta \nabla f(x_{t+1}). \\[-20pt] \nonumber
\end{align}
In words, instead of minimizing $f$ directly, PPA minimizes $f$ with an additional quadratic term. 
This small change brings a major advantage to PPA: if $f$ is convex, the added quadratic term can make the problem strongly convex; if $f$ is non-convex, PPA can make it convex \citep{ahn_proximal_2020}.
Thanks to this conditioning property, PPA exhibits different behavior compared to GD in the deterministic setting.
\cite{guler_convergence_1991} proved that for a convex function $f$, PPA satisfies: 
\begin{align}
\label{eq:ppm-conv-rate-guller}
    f(x_T) - f(x^\star) \leq O \left( \tfrac{1}{\sum_{t=1}^T \eta_t} \right),  \\[-16pt] \nonumber
\end{align}
after $T$ iterations.
By setting the step size $\eta_t$ to be large, PPA can converge ``arbitrarily'' fast. 

PPA was soon considered in the stochastic setting. 
In \cite{ryu_stochastic_2017}, a stochastic version of PPA, dubbed as stochastic proximal iterations (SPI), was analyzed, where an approximation of $f$ using a single data $f_i$ was considered.
The same algorithm was (statistically) analyzed under the name of implicit stochastic gradient descent (ISGD) \citep{toulis_statistical_2014, toulis_asymptotic_2017}, and was extended to the Robbins-Monro procedure in \cite{toulis_proximal_2021}. 
Similar algorithms were analyzed recently in \cite{asi_stochastic_2019, asi_minibatch_2020} where each $f_i$ was further approximated by simpler surrogate functions. 
These works generally indicate that, in the asymptotic regime, SGD and SPI/ISGD have the same convergence behavior; but in the non-asymptotic regime, SPI/ISGD outperforms SGD due to numerical stability provided by utilizing proximal updates.

\paragraph{Accelerated PPA.}
Accelerated PPA was first proposed in deterministic setting in \cite{guler_new_1992}, where Nesterov's acceleration was applied \emph{after} solving the proximal step in \eqref{eq:ppa}. This yields the convergence rate of the form: 
\vspace{-0.4cm}
\begin{align}
    f(x_T) - f(x^\star) \leq O \left( \tfrac{1}{ \left( \sum_{t=1}^T \sqrt{\eta_t} \right)^2} \right), \label{eq:ppm-acc-conv-rate-guller}
\end{align}
which is faster than the rate in \eqref{eq:ppm-conv-rate-guller}. 
This bound is based on Nesterov's momentum schedules, but does not study the effect in stability different tuning pairs $(\eta, \beta)$ might have. 
Moreover, as can be seen in \eqref{eq:ppm-conv-rate-guller}, PPA can already achieve arbitrarily fast convergence, given it is implemented exactly.
Hence, following works focused on studying the conditions under which the proximal step in \eqref{eq:ppa} can be computed inexactly, while still exhibiting some acceleration \citep{lin_universal_2015, lin_catalyst_2018}; similar analyses were later extended to the stochastic setting in \cite{kulunchakov_generic_2019}.

\cite{chadha_accelerated_2021} and \cite{deng_minibatch_2021} also considered accelerated SPPA.
Both of these works apply a convoluted 2- or 3-step Nesterov's procedure after the proximal step, where $f_i$ was further approximated with auxiliary functions. 
Yet, stability arguments via proximal updates are less apparent due to the auxiliary functions, requiring specific step size and momentum schedules, which might involve an additional one-dimensional optimization per iteration; see also Theorem~\ref{thm:init-discount-condition}.
A summary of these algorithms is provided in Table~\ref{tab:algos}.

\paragraph{Intuition of SPPAM in \eqref{eq:acc-stoc-ppa}.}

In contrast to the aforementioned works, we include Polyak's momentum \citep{polyak_methods_1964}
directly to SPPA, yielding \eqref{eq:acc-stoc-ppa}. 
Apart from the similarity between SPPAM in \eqref{eq:acc-stoc-ppa} and SGDM in \eqref{eq:sgdm}, SPPAM shares the same geometric intuition as Polyak's momentum for SGDM.
Disregarding the stochastic errors, the update in  \eqref{eq:acc-stoc-ppa} follows from the solution of: 
\vspace{-0.1cm}
\begin{align*}
    \arg \min_{x \in \mathbb{R}^p} \left\{ f(x) + \tfrac{1}{2\eta} \|x-x_t\|_2^2 - \tfrac{\beta}{\eta} \langle x_t - x_{t-1}, x \rangle \right\}. \\[-20pt] \nonumber
\end{align*}

We can get a sense of the behavior of SPPAM from this expression.
First, for large $\eta$, the algorithm is minimizing the original $f.$ 
For small $\eta$, the algorithm not only tries to stay local by minimizing the quadratic term, but also tries to minimize 
$-\frac{\beta}{\eta} \langle x_t - x_{t-1}, x \rangle$.
By the definition of inner product, this means that $x$, on top of minimizing $f$ and staying close to $x_t$, also tries to move along the direction from $x_{t-1}$ to $x_t$. This intuition aligns with that of Polyak's momentum.

\section{The Quadratic Model Case}
\label{sec:quad-model}

\begin{figure*}[!h]
    \centering
    \includegraphics[width=0.195\linewidth]{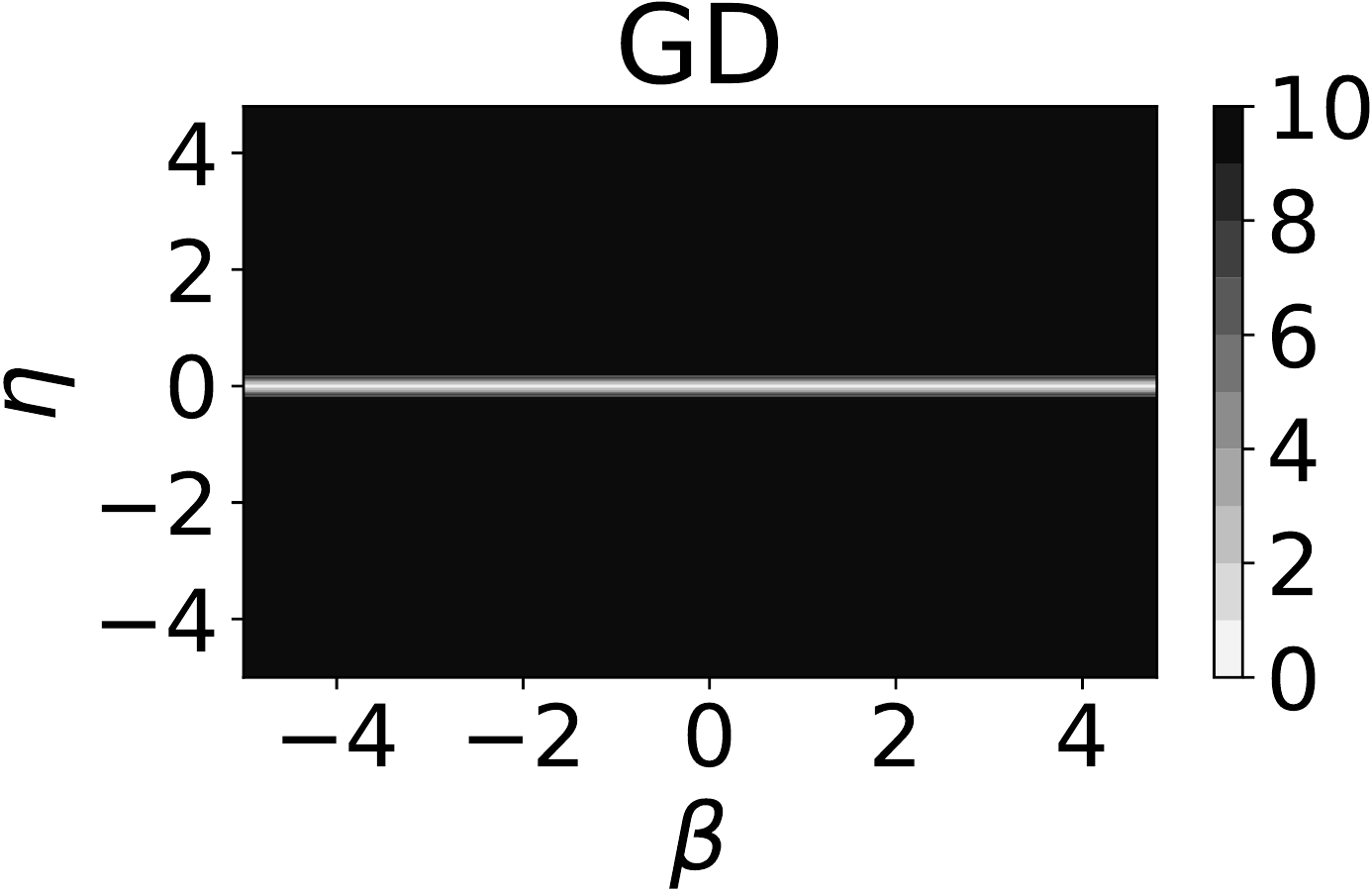}
    \includegraphics[width=0.195\linewidth]{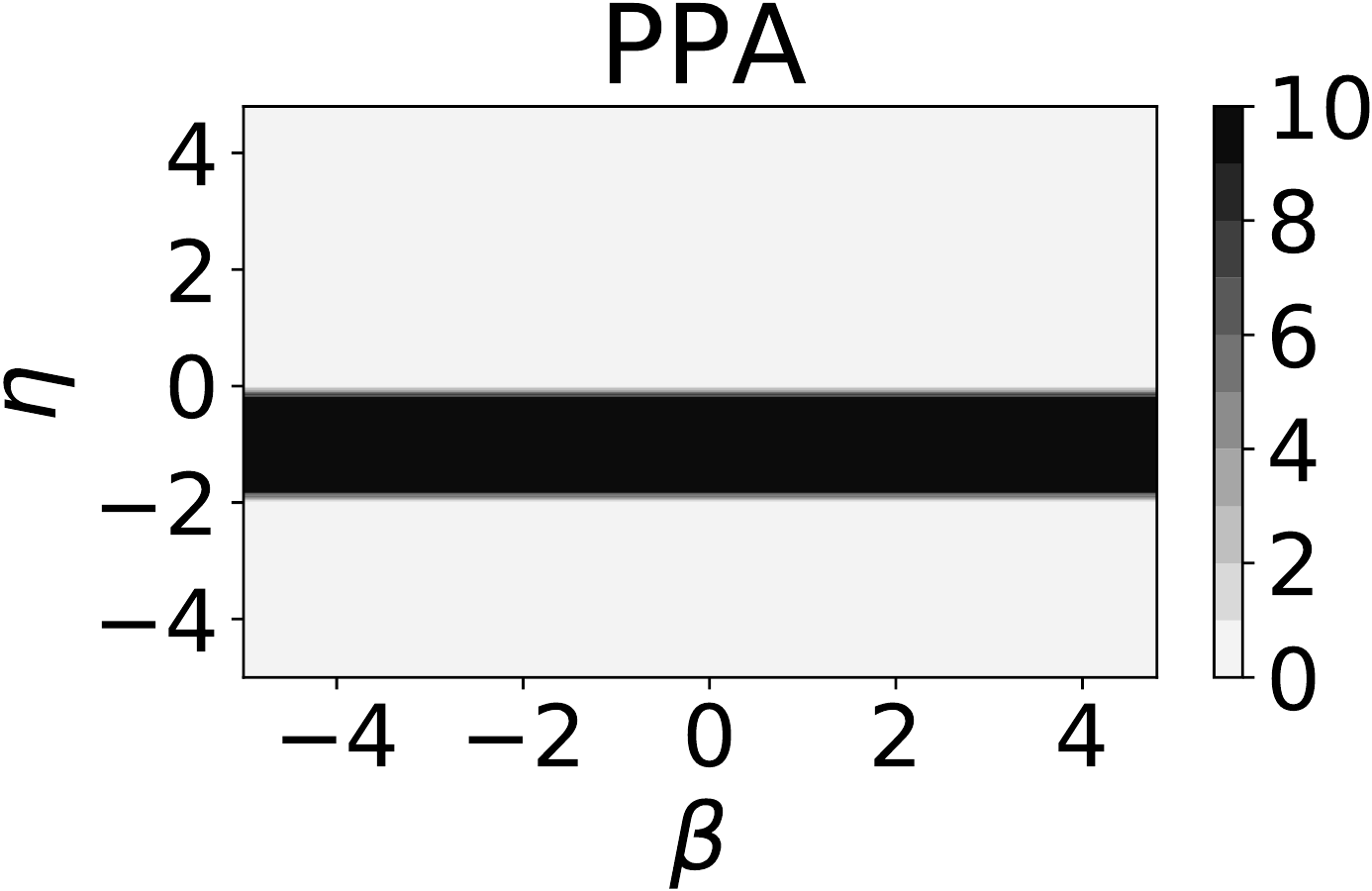}
    \includegraphics[width=0.195\linewidth]{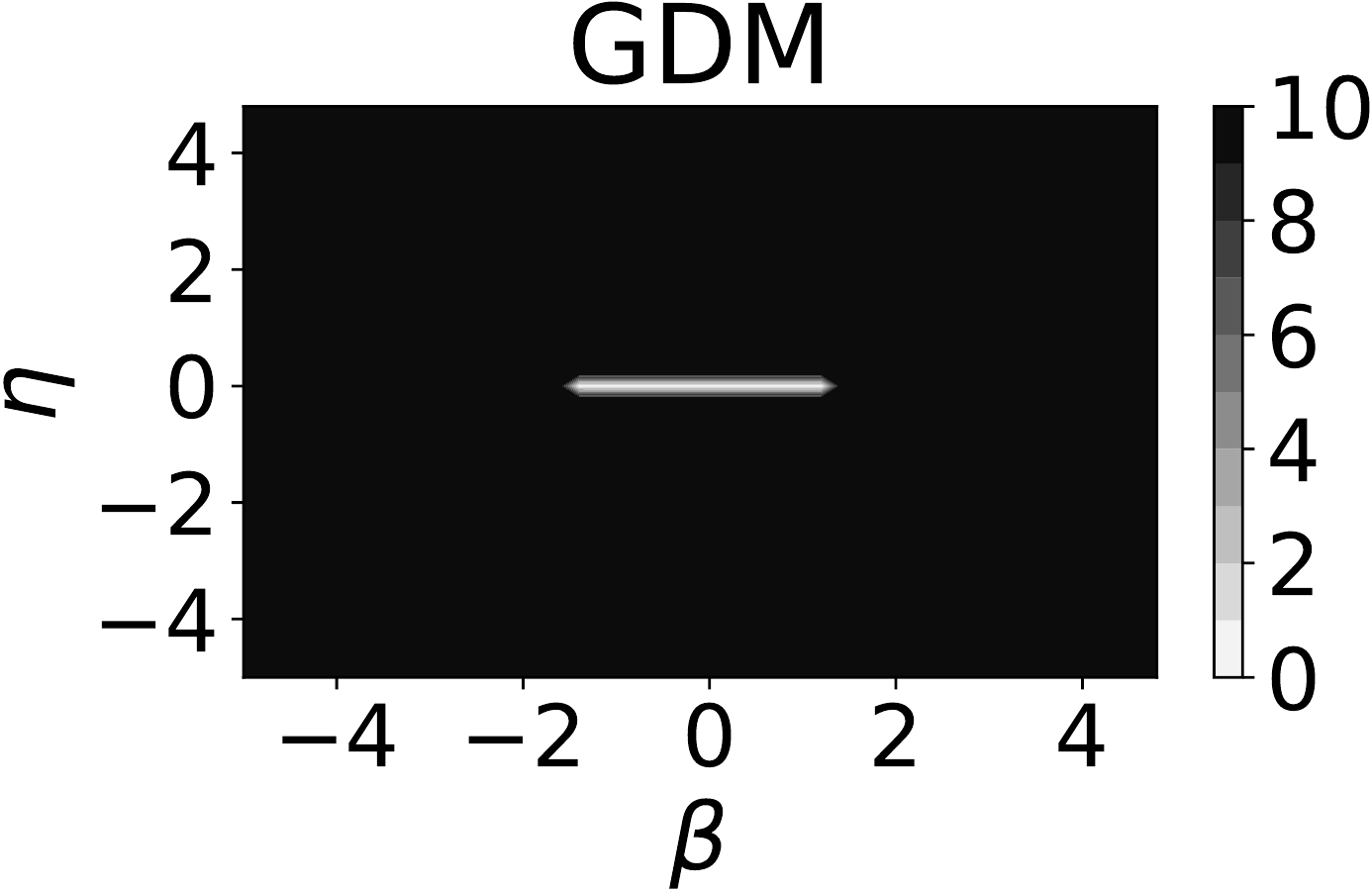}
    \includegraphics[width=0.195\linewidth]{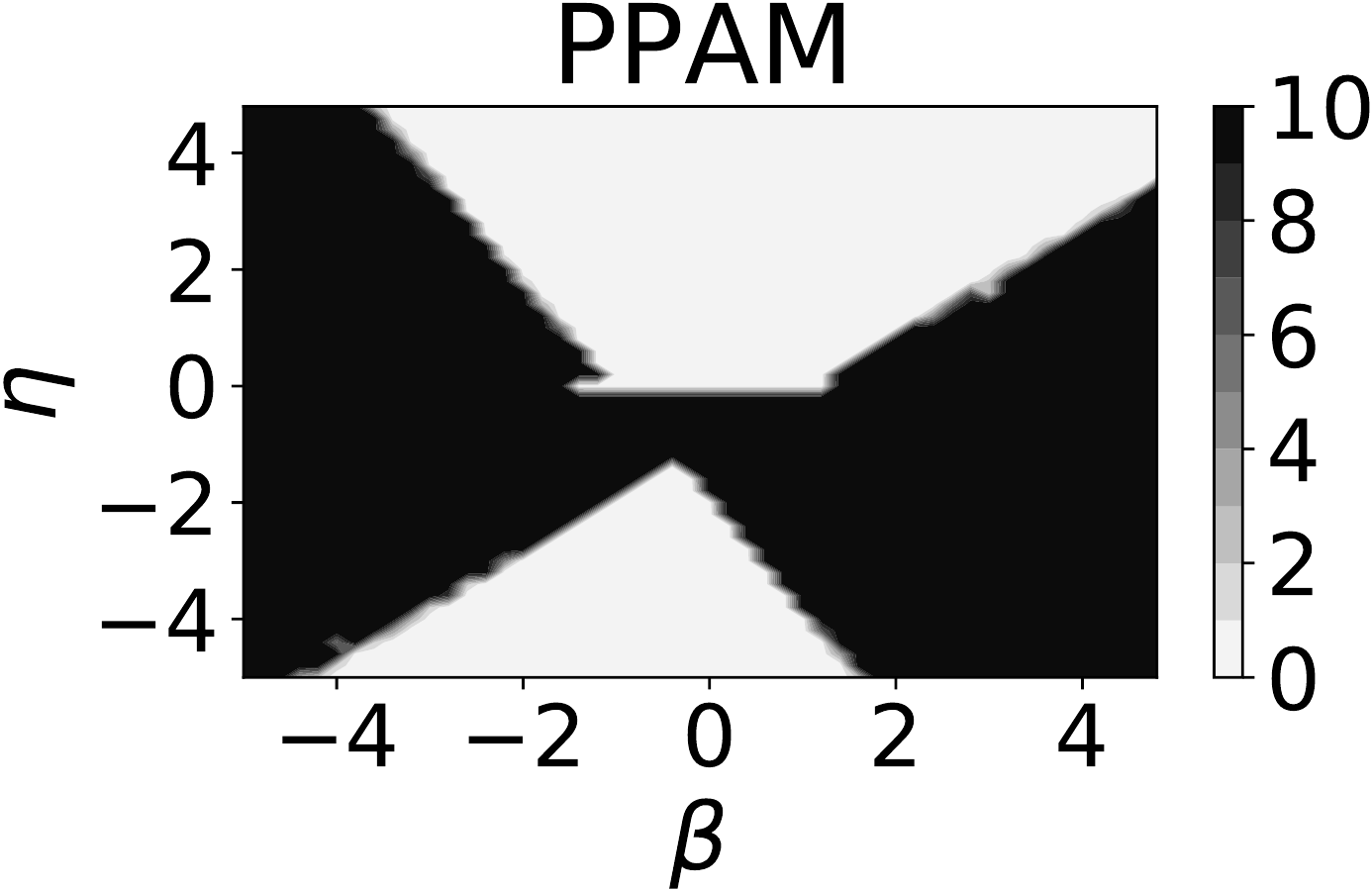}
    \includegraphics[width=0.195\linewidth]{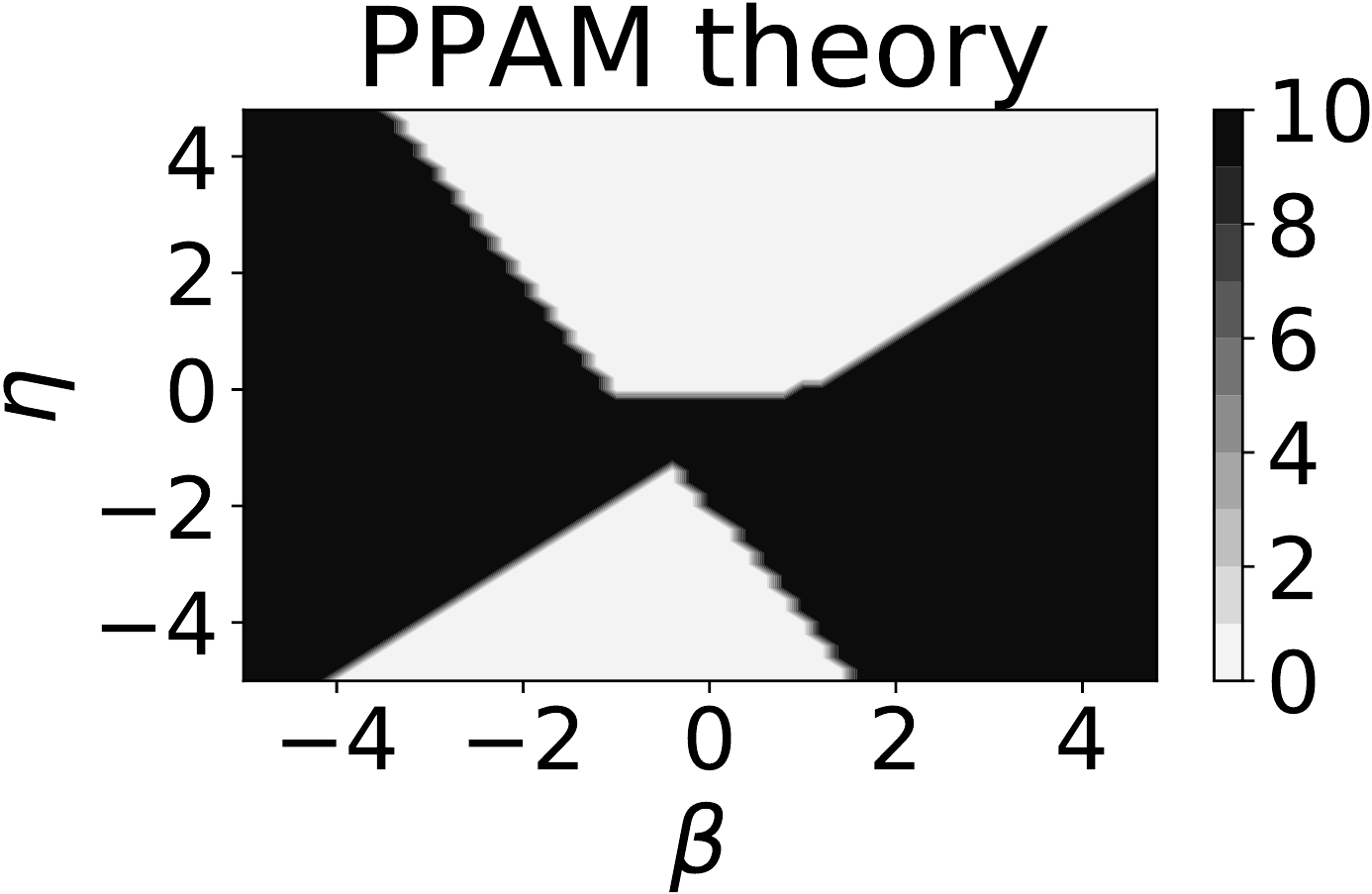}
    \vspace{-0.6cm}
    \caption{
    We generate $A \in \mathbb{R}^{p\times p}$ and $b, x^\star \in\mathbb{R}^p$ from $\mathcal{N}(0, I)$, where $p=100$ and the condition number of $A$ is 10. We sweep $\eta$ and $\beta$ from $-5$ to $5$, with $0.2$ interval. We plot the accuracy $\|x_t - x^\star\|_2^2$ after 100 iterations, with the maximum replaced by 10. 
    }
    \label{fig:conv-region}
    \vspace{-0.4cm}
\end{figure*}

For simplicity, we first consider the convex quadratic optimization problem under the deterministic setting. 
Specifically, we consider the objective function: 
\vspace{-0.1cm}
\begin{align}
\label{eq:obj-quad}
    f(x) = \frac{1}{2} x^\top A x - b^\top x, 
     \\[-21pt] \nonumber
\end{align}
where $A \in \mathbb{R}^{p\times p}$ is positive semi-definite with eigenvalues $\left[ \lambda_1, \dots, \lambda_p \right]$. 
Under this scenario, 
we can study how the step size $\eta$ and the momentum $\beta$ affect each other, by deriving exact conditions that lead to convergence for each algorithm.
The comparison list includes 
gradient descent (GD), gradient descent with momentum (GDM), the PPA, and PPA with momentum (PPAM).
Propositions~\ref{prop:gd} and \ref{prop:gdm} for GD and GDM are from {{\cite{goh2017why}}}, and included for completeness. 
Proofs for PPA and PPAM in Propositions~\ref{prop:ppa} and \ref{prop:ppam} can be found in the extended version of this work \citep{kim2021convergence}.

\begin{proposition}[GD \citep{goh2017why}] \label{prop:gd}
  To minimize \eqref{eq:obj-quad} with gradient descent, the step size $\eta$ needs to satisfy $0 < \eta < \frac{2}{\lambda_i}~~\forall i$, where $\lambda_i$ is the $i$-th eigenvalue of $A$.
\end{proposition}

\begin{proposition}[PPA] \label{prop:ppa}
  To minimize \eqref{eq:obj-quad} with PPA, the step size $\eta$ needs to satisfy $\left| \frac{1}{1+\eta \lambda_i} \right| < 1~~\forall i$.
\end{proposition}

\begin{proposition}[GDM \citep{goh2017why}] \label{prop:gdm}
  To minimize \eqref{eq:obj-quad} with gradient descent with momentum, the step size $\eta$ needs to satisfy $0 < \eta \lambda_i < 2 + 2\beta$ ~ $\forall i$, where $0 \leq \beta < 1.$ 
\end{proposition}

\begin{proposition}[PPAM] \label{prop:ppam}
    Let $\delta_i = \left( \frac{\beta+1}{1+\eta \lambda_i} \right)^2 - \frac{4\beta}{1+\eta \lambda_i}.$
    To minimize \eqref{eq:obj-quad} with PPAM, the step size $\eta$ and momentum $\beta$ need to satisfy $~\forall i$: 
    \vspace{-0.1cm}
    \begin{itemize}[leftmargin=0.8cm]
        \item $\eta > \frac{\beta-1}{\lambda_i},$  \quad \quad \quad \quad \quad ~if ~$\delta_i \leq 0$;
        \vspace{-2mm}
        \item $\frac{\beta+1}{1+\eta \lambda_i} + \sqrt{\delta_i} < 2,$ ~ \quad if ~$\delta_i > 0$ ~and~ $\frac{\beta+1}{1+\eta \lambda_i} \geq 0$;
        \vspace{-2mm}
        \item $\frac{\beta+1}{1+\eta \lambda_i} - \sqrt{\delta_i} > -2,$  \quad otherwise.
    \end{itemize}
\end{proposition}
Given the above propositions, we can study the stability with respect to the step size $\eta$ and the momentum $\beta$ for the considered algorithms. 
Numerical simulations support the above propositions, and are illustrated in Figure~\ref{fig:conv-region}, matching the theoretical conditions exhibited above. 
In particular, for GD ({\sf1st}), only a small range of step sizes $\eta$ leads to convergence (small white band); this ``white band'' corresponds to the restriction that $\eta$ has to satisfy $\eta < \tfrac{2}{\lambda_i}$ for all $i$. 
On the other hand, PPA/IGD ({\sf2nd}) converges in much wider choices of $\eta$; this is apparent from Proposition \ref{prop:ppa}, since $\left| \frac{1}{1+\eta \lambda_i} \right|$ can be arbitrarily small for larger values of $\eta$.
GDM ({\sf3rd}) requires both $\eta$ and $\beta$ to be in a small region to converge, following Proposition \ref{prop:gdm}.
Finally, PPAM ({\sf4th}) converges in much wider choices of $\eta$ and $\beta$; e.g., the conditions in Proposition \ref{prop:ppam} define different regions of the pair $(\eta, \beta)$ that lead to convergence, some of which set both $\eta$ and $\beta$ to be negative. 
Note that the empirical convergence region of PPAM almost exactly matches the theoretical region that leads to convergence in Proposition~\ref{prop:ppam} ({\sf5th}).
In the remainder of the paper, we study how such pattern translates to a general strongly convex function $f,$ with stochasticity.

\section{Theoretical Results}
\label{sec:theory}

In this section, we theoretically characterize the convergence and stability behavior of SPPAM.\footnote{All proofs are provided in the extended version of this work \citep{kim2021convergence}.}
We follow the stochastic errors of PPA, as set up in \cite{toulis_proximal_2021}; we can thus express \eqref{eq:acc-stoc-ppa} as:\footnote{
$x_{t+1}^+$ is an auxiliary intermediate variable that is only used for the analysis.} 
\begin{align*}
    x_{t+1}^+ 
    &= x_t - \eta \nabla f(x_{t+1}^+) + \beta (x_t - x_{t-1})  \\
    x_{t+1} &= x_{t+1}^+ - \eta  \varepsilon_{t+1}.
\end{align*}
We further assume the following: \vspace{-0.1cm}
\begin{assumption} \label{assump:str-convex}
    $f$ is a $\mu$-strongly convex function. That is, for some fixed $\mu > 0$ and for all $x$ and $y$,
    \begin{align*} 
    \langle \nabla f(x)-\nabla f(y), x - y \rangle \geq \mu \|x-y \|_2^2.
    \end{align*}
\end{assumption}
\vspace{-0.3cm}
\begin{assumption} \label{assump:bounded-var}
    There exists fixed $\sigma^2 > 0$ such that, given the natural filtration $\mathcal{F}_{t-1},$
    \begin{align*}
        \mathbb{E}\left[ \varepsilon_t \mid \mathcal{F}_{t-1} \right] = 0 ~~\text{and}~~
        \mathbb{E}\left[ \| \varepsilon_t \|^2 \mid \mathcal{F}_{t-1} \right] \leq \sigma^2
        ~~\text{for all}~~ t. 
        \\[-18pt] \nonumber
    \end{align*}
\end{assumption}

We now study whether SPPAM enjoys faster convergence than SPPA in \eqref{eq:stoc-ppa}.
We start with the iteration invariant bound:
\begin{theorem} 
\label{thm:onestep-acc-stoc-prox}
  For $\mu$-strongly convex $f$, SPPAM 
satisfies the following iteration invariant bound: 
\vspace{-0.1cm}
\begin{align} \label{eq:onestep-acc-stoc-prox}
    \ex{ \|x_{t+1 }- x^\star\|_2^2 } &\leq \tfrac{4}{(1+\eta \mu)^2} \ex{ \|x_t - x^\star \|_2^2} + \tfrac{4\beta^2}{(1+\eta \mu)^2\left( 4-(1+\beta)^2 \right)} \ex{ \|x_{t-1} - x^\star \|_2^2 } + \eta^2 \sigma^2. \\[-10pt] \nonumber 
\end{align}
\end{theorem}
Notice that all terms --except the last one-- are divided by $(1+\eta \mu)^2.$ 
Thus, large step size $\eta$ help convergence (to a neighborhood), reminiscent of the convergence behavior of PPA in \eqref{eq:ppm-conv-rate-guller}. 
Based on \eqref{eq:onestep-acc-stoc-prox}, 
we can write the following $2\times 2$ system that characterizes the progress of SPPAM: 
\begin{align}
\label{eq:two-by-two-onestep}
\begin{bmatrix}
  \ex{ \|x_{t+1} - x^\star \|_2^2 } \\
  \ex{ \|x_t - x^\star \|_2^2 }
\end{bmatrix}
&\leq
\underbrace{
\begin{bmatrix}
  \frac{4}{(1+\eta \mu)^2} & \frac{4\beta^2}{(1+\eta \mu)^2\left( 4-(1+\beta)^2 \right)} \\
  1 & 0
\end{bmatrix}
}_A
\cdot
\begin{bmatrix}
  \ex{ \|x_t - x^\star \|_2^2 } \\
  \ex{ \|x_{t-1} - x^\star \|_2^2 }
\end{bmatrix}
+ 
\begin{bmatrix}
  \eta^2 \sigma^2 \\
  0
\end{bmatrix}.
\end{align}

Therefore, the spectral radius of the contraction matrix $A$ (asymptotically) determines the convergence rate, as in {\citep{goh2017why}}. This is summarized in the following lemma:
\begin{lemma} \label{lem:SPPAM-contraction}
The maximum eigenvalue of $A$,
which determines the convergence rate of SPPAM, is:
\begin{align}
\label{eq:acc-stoc-ppa-conv-rate}
    \tfrac{2}{(1+\eta \mu)^2} + \sqrt{ \tfrac{4}{(1+\eta \mu)^4} + \tfrac{4\beta^2}{(1+\eta \mu)^2(4 - (1+\beta)^2)}}.
\end{align}
\end{lemma}
Notice the one-step contraction factor in \eqref{eq:acc-stoc-ppa-conv-rate} is of order $O(1/\eta^2),$ exhibiting acceleration compared to that of SPPA for strongly convex objectives \citep{toulis_proximal_2021}: $1/(1+2\eta \mu) \approx O(1/\eta).$ 
However, due to the additional terms, it is not immediately obvious when SPPAM enjoys faster contraction than SPPA.
We thus characterize this condition more precisely in the following corollary:
\begin{corollary} \label{cor:acc-condition}
For $\mu$-strongly convex $f$, SPPAM 
enjoys smaller contraction factor than SPPA if: 
\begin{small}
\begin{align*}
    \frac{4\beta^2}{4 - (1+\beta)^2} <  \frac{\eta^2 \mu^2 - 6\eta\mu - 3}{(1+\eta \mu)^2}. 
\end{align*}
\end{small}
\end{corollary}
In words, for a fixed step size $\eta$ and given a strongly convex parameter $\mu$, there is a range of momentum parameters $\beta$ that exhibits acceleration compared to SPPA. 

\begin{remark}
In contrast to (stochastic) gradient method analyses in convex optimization, where acceleration is usually shown by improving the dependency on the condition number from $\kappa = \tfrac{L}{\mu}$ to $\sqrt{\kappa},$ such a claim can hardly be made for stochastic proximal point methods. This is also the case in deterministic setting; see \eqref{eq:ppm-conv-rate-guller} and \eqref{eq:ppm-acc-conv-rate-guller}. 
As shown in Theorem~\ref{thm:onestep-acc-stoc-prox}, our convergence analysis of SPPAM does not depend on $L$-smoothness at all. This robustness of SPPAM is also confirmed in numerical simulations in Section~\ref{sec:experiments}, where SPPAM exhibits the fastest convergence rate, virtually independent of the different settings considered.
\end{remark}

We now formalize the convergence behavior of SPPAM.
In particular, we characterize the condition that leads to an exponential discount of the initial conditions.
By unrolling the recursion of SPPAM in \eqref{eq:two-by-two-onestep} for $T$ iterations, we obtain: 
\begin{align}
\begin{bmatrix}
  \ex{ \|x_T - x^\star \|_2^2 } \\
  \ex{ \|x_{T-1} - x^\star \|_2^2 } 
\end{bmatrix}
&\leq 
A^T
\cdot
\begin{bmatrix}
   \|x_0 - x^\star \|_2^2 \\
   \|x_{-1} - x^\star \|_2^2 
\end{bmatrix} 
+ 
\left(
\sum_{i=1}^{T-1}A^i\right)
\begin{bmatrix}
   1 \\
   0
\end{bmatrix} 
\eta^2 \sigma^2. \nonumber 
\end{align}  
%
It is clear from the above that the convergence is determined by $A^T$ and $\left(\sum_{i=1}^{T-1} A^i\right),$ where $A$ was defined in \eqref{eq:two-by-two-onestep}. Our next theorem derives convergence to a neighborhood based on the spectrum of these quantities, akin to \citet[Theorem 1]{assran_convergence_2020} and \citet[Theorem 3]{toulis_proximal_2021}.

\begin{theorem} \label{thm:lin-conv}
For $\mu$-strongly convex $f$, assume SPPAM 
is initialized with $x_0 = x_{-1}$. Then, after $T$ iterations, we have:
\begin{align}  \label{eq:Tstep-acc-stoc-prox}
   &\ex{ \|x_T - x^\star \|_2^2 } 
    \leq  \frac{2 \sigma_1^T}{\sigma_1 - \sigma_2} 
    \left(  \left( \|x_0-x^\star \|_2^2 + \tfrac{\eta^2 \sigma^2}{1-\theta}\right) \cdot (1+\theta) \right)   
    +
    \frac{\eta^2 \sigma^2}{1-\theta}, 
\end{align} 
where 
$\theta=\frac{4}{(1+\eta\mu)^2} + \tfrac{4\beta^2}{(1+\eta\mu)^2(4-(1+\beta)^2)}.$
Here, $\sigma_{1, 2}$ are the eigenvalues of $A$, and
\begin{align} \label{eq:discount-init}
    \tfrac{2 \sigma_1^T}{\sigma_1 - \sigma_2}  =  \tau^{-1} 
    \cdot
    \left( \tfrac{2}{(1+\eta\mu)^2} + \tau \right)^T
    \quad \text{with} \quad 
    \tau = \sqrt{ \tfrac{4}{(1+\eta \mu)^4} + \tfrac{4\beta^2}{(1+\eta \mu)^2(4 - (1+\beta)^2)}}.
\end{align}
\end{theorem}

The above theorem states that the term in \eqref{eq:discount-init} determines the discounting rate of the initial conditions.
In particular, the condition that leads to an exponential discount 
of the initial conditions 
is characterized by the following theorem:

\begin{theorem} \label{thm:init-discount-condition}
Let the following condition hold:
\begin{align} \label{eq:init-discount-condition}
    \tau = \sqrt{ \tfrac{4}{(1+\eta \mu)^4} + \tfrac{4\beta^2}{(1+\eta \mu)^2(4 - (1+\beta)^2)}} < \tfrac{1}{2}.
\end{align}
Then, for $\mu$-strongly convex $f$, initial conditions of SPPAM exponentially discount: i.e., in \eqref{eq:Tstep-acc-stoc-prox},  
\begin{align*}
    \tfrac{2 \sigma_1^T}{\sigma_1 - \sigma_2}  =
      \tau^{-1} \cdot   \left( \tfrac{2}{(1+\eta\mu)^2} + \tau \right)^T 
    =C^T, \quad\text{where}\quad C \in (0, 1).
\end{align*}
\end{theorem}

\begin{remark}
The condition in \eqref{eq:init-discount-condition} is much easier to satisfy than accelerated SGD. 
E.g., as described below, the required condition for accelerated SGD to converge linearly to a neighborhood in strongly convex quadratic objective relies on knowing $\eta = \frac{1}{L}$ and momentum $\beta = \frac{\sqrt{\kappa}-1}{\sqrt{\kappa}+1}$ \citep{assran_convergence_2020}, where both $L$ and $\kappa$ are unknown in practice. 
While this is also true for SPPAM (i.e., $\mu$ is an unknown quantity), \eqref{eq:init-discount-condition} suggests that one can essentially set $\eta$ sufficiently large to ensure the exponential discount, even without knowing $\mu$ exactly.
\end{remark}

\begin{remark} \label{remark:stability-comparison}
Other works that study variants of accelerated SPPA \citep{kulunchakov_generic_2019, chadha_accelerated_2021} still require specific choices of step size and momentum (e.g., $\eta_t = \frac{1}{L + c_0\sqrt{t+1}}$, $\beta_t = \frac{2}{t+2}$ for the latter; see Table~\ref{tab:algos} for the former), similarly to accelerated SGD. 
\end{remark}

To provide more context of the condition in Theorem~\ref{thm:init-discount-condition}, we make an ``unfair" comparison of  \eqref{eq:init-discount-condition}, which holds for general strongly convex $f,$ to the condition that accelerated SGD requires for strongly convex \emph{quadratic objective} in \eqref{eq:obj-quad}.
\citet[Theorem 1]{assran_convergence_2020} show that Nesterov's accelerated SGD converges to a neighborhood at a linear rate for strongly convex quadratic objective if 
    $\max\{ \rho_\mu(\eta, \beta),~\rho_L(\eta, \beta) \} < 1$, 
where $\rho_\lambda(\eta, \beta)$ for $\lambda \in \{\mu, L\}$ is defined as:
\vspace{-0.1cm}
\begin{align} \label{eq:sgdm-spectral-rad}
\rho_\lambda(\eta, \beta) = 
\begin{cases}
\frac{|(1+\beta)(1-\eta \lambda)|}{2} + \frac{\sqrt{\Delta_\lambda}}{2} & \text{if}~\Delta_\lambda \geq 0, \\
\sqrt{\beta (1-\eta \lambda)} & \text{otherwise},
\end{cases} 
\end{align}
with $\Delta_\lambda = (1+\beta)^2 (1-\eta \lambda)^2 - 4\beta(1-\eta \lambda)$.
This condition for convergence can thus be divided into three cases, depending on the range of $\eta \lambda$. 
Define $\psi_{\beta, \eta, \lambda} = (1 + \beta)(1 - \eta \lambda)$. Then: 
\begin{align*}
\begin{cases}
 \eta \lambda \geq 1, &  \text{Converges if }-\psi_{\beta, \eta, \lambda} + \sqrt{\Delta_\lambda} < 2, \vspace{1mm} \\ 
 \frac{(1-\beta)^2}{(1+\beta)^2} \leq \eta \lambda < 1,
& \text{Always converges}, \vspace{1mm} \\
 \eta \lambda < \frac{(1-\beta)^2}{(1+\beta)^2}, &\text{Converges if }\psi_{\beta, \eta, \lambda} + \sqrt{\Delta_\lambda}< 2 .
\end{cases}
\end{align*}

Now, consider the standard momentum value $\beta = 0.9.$ 
For the first case, the convergence requirement translates to 
$1 \leq \eta \lambda \leq \tfrac{24}{19}.$ 
The second range is given by $\frac{1}{361} \leq \eta \lambda < 1$. 
The third condition is lower bounded by 2 for $\beta=0.9,$ leading to divergence.
Combining, accelerated SGD requires $0.0028 \approx \frac{1}{361} \leq \eta \lambda \leq \frac{24}{19} \approx 1.26$ to converge for strongly convex quadratic objectives, set aside that this bound has to satisfy for (unknown) $\mu$ or $L$.
Albeit an unfair comparison, for general strongly convex objective, \eqref{eq:init-discount-condition} becomes $\eta \mu > 4.81$ for $\beta=0.9.$
Even though $\mu$ is unknown, one can see this condition is easy to satisfy, by using a sufficiently large step size $\eta$.

\section{Experiments}
\label{sec:experiments}

In this section, we perform numerical experiments to study the convergence behaviors of SPPAM, SPPA, SGDM, and SGD, using generalized linear models (GLM) {\cite{nelder1972generalized}}. 
Let $b_i \in \mathbb{R}$ be the label, $a_i \in \mathbb{R}^{p}$ be the features, and $x^\star \in \mathbb{R}^p$ be the model parameter of interest. GLM assumes that $b_i$ follows an exponential family distribution: $b_i \mid a_i \sim \exp \left( \frac{\gamma b_i - c_1(\gamma)}{\omega} c_2(b_i, \omega)  \right).$
Here, $\gamma=\langle a_i, x^\star \rangle$ is the linear predictor, $\omega$ is the dispersion parameter related to the variance of $b_i$, and $c_1(\cdot)$ and $c_2(\cdot)$ are known real-valued functions. 
GLM subsumes a wide family of models including linear, logistic, and Poisson regressions. Different models connects the linear predictor $\gamma=\langle a_i, x^\star \rangle$ through different \textit{mean} functions $h(\cdot)$.
We focus on linear and Poisson regression models, where mean functions are defined respectively as $h(\gamma) = \gamma$ and $h(\gamma) = e^\gamma$.
The former is an ``easy'' case, where objective is strongly convex, satisfying Assumption~\ref{assump:str-convex}. The latter is a ``hard'' case with non-Lipschitz continuous gradients, where SGD and SGDM are expected to suffer.

{\cite{toulis_statistical_2014}} introduced an efficient, exact implementation of SPPA for GLM. We adapt this procedure to SPPAM in \eqref{eq:acc-stoc-ppa}; see Algorithm~\ref{alg:sppam-glm}. Its derivation can be found in \cite{kim2021convergence}.

\begin{wrapfigure}{L}{0.5\textwidth}
\vspace{-0.5cm}
\begin{minipage}{0.5\textwidth}
  \begin{algorithm}[H]
    \caption{SPPAM for GLM}\label{alg:sppam-glm}
    \For{$t = 1, 2, \dots$}{
  Sample $i_t \sim$ Unif$(1, n)$ \\ 
  $r_t \gets \eta (b_{i_t} - h(\langle a_{i_t}, x_{t-1} \rangle)$\\
  $B_t \gets [0, r_t]$ \\
  \If{$r_t \leq 0$} {
  $B_t \gets [r_t, 0]$ 
  }
  $
  \begin{aligned}
    \xi_t = &\eta \big[ b_{i_t}  - h( (1+\beta)  \langle a_{i_t}, x_{t-1} \rangle \\
  & - \beta \langle a_{i_t}, x_{t-2} \rangle + \xi_t \cdot \|a_{i_t}\|_2^2 ) \big],~\xi_t \in B_t
  \end{aligned}
  $ \\
  $x_t \gets x_{t-1} + \xi_t \cdot a_{i_t} + \beta(x_{t-1} - x_{t-2})$ 
}
  \end{algorithm}
\end{minipage}
\vspace{-0.3cm}
\end{wrapfigure}

\begin{figure*}[!h]
    \centering
    \includegraphics[scale=0.32]{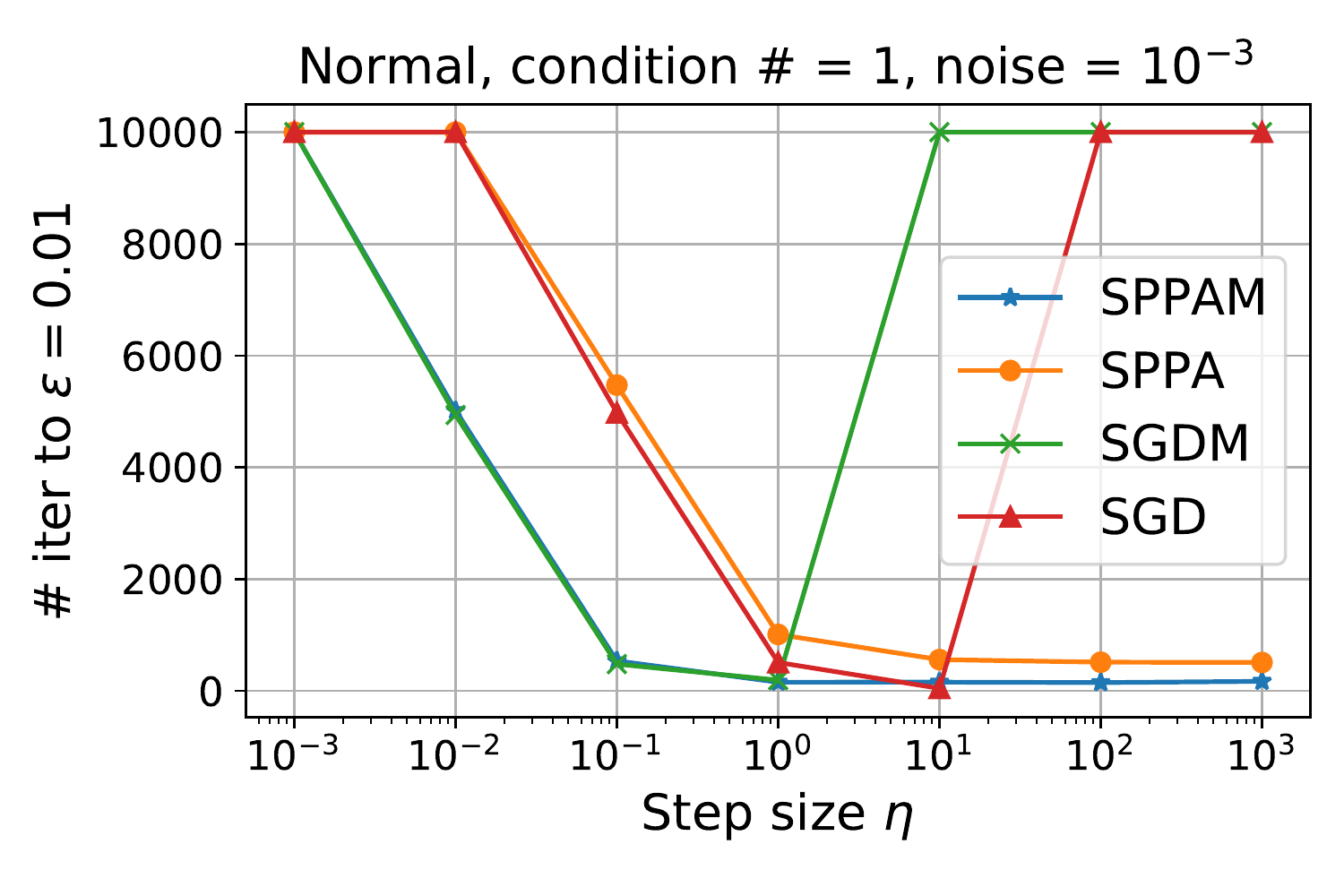}
    \includegraphics[scale=0.32]{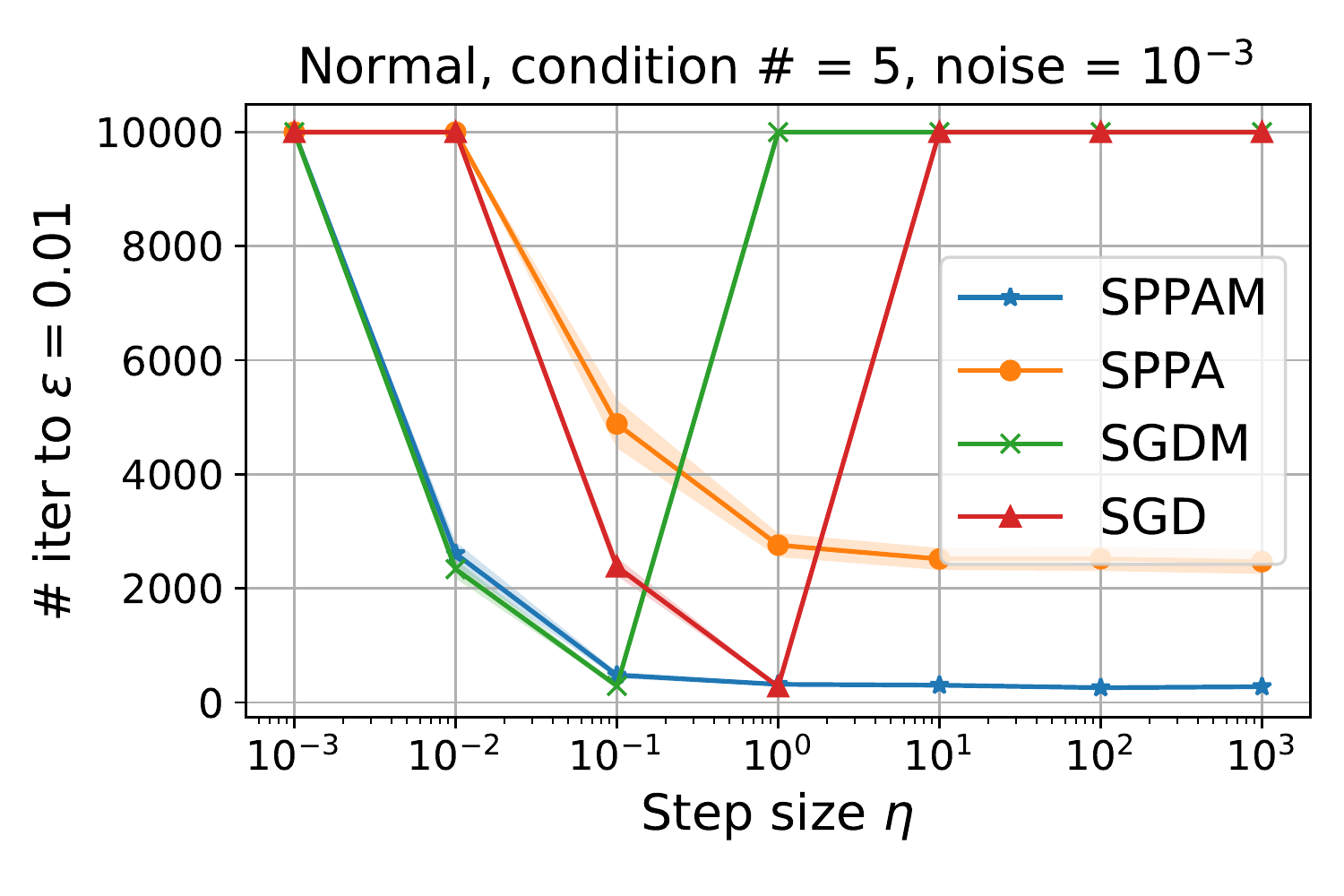}
    \includegraphics[scale=0.32]{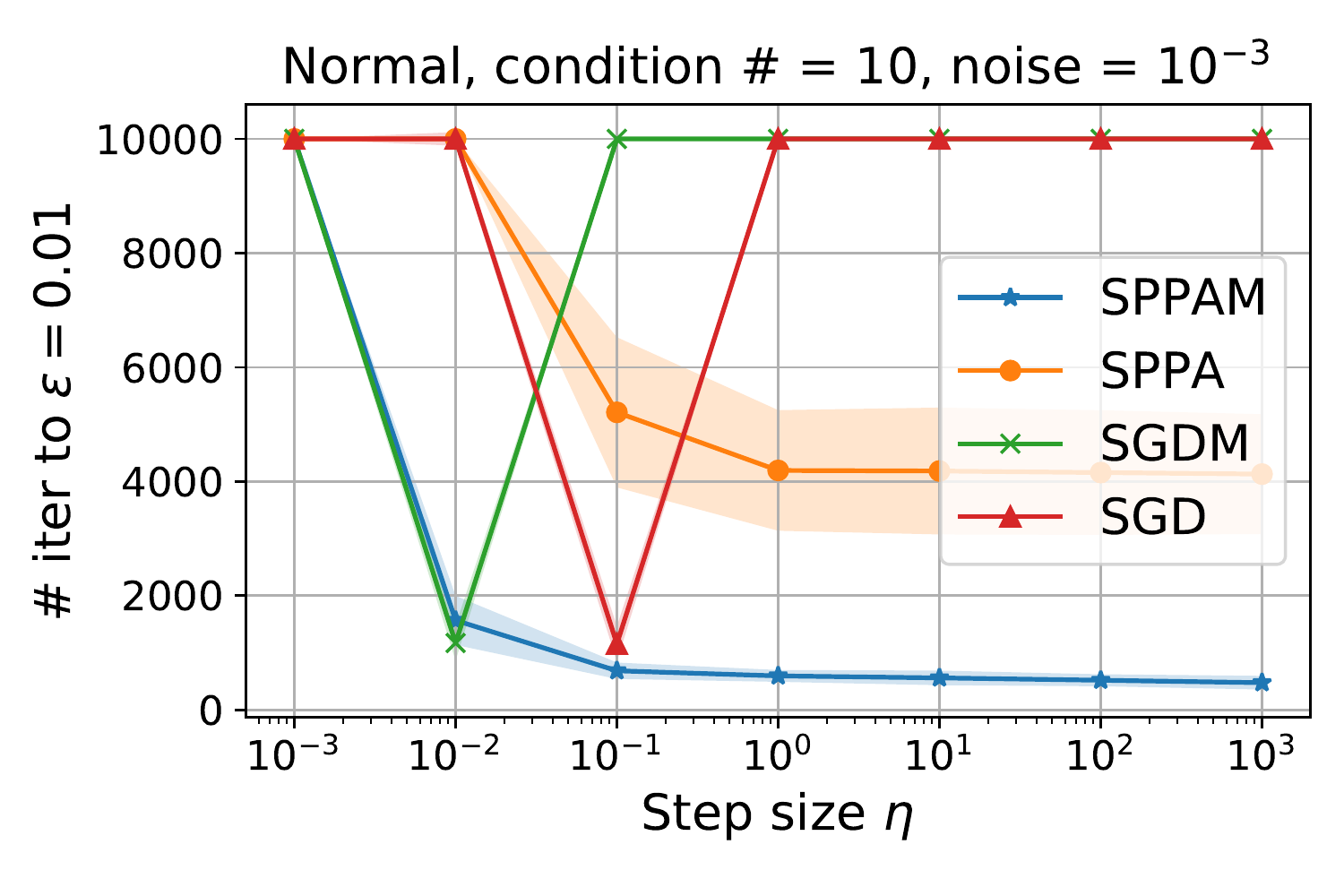}
    \includegraphics[scale=0.32]{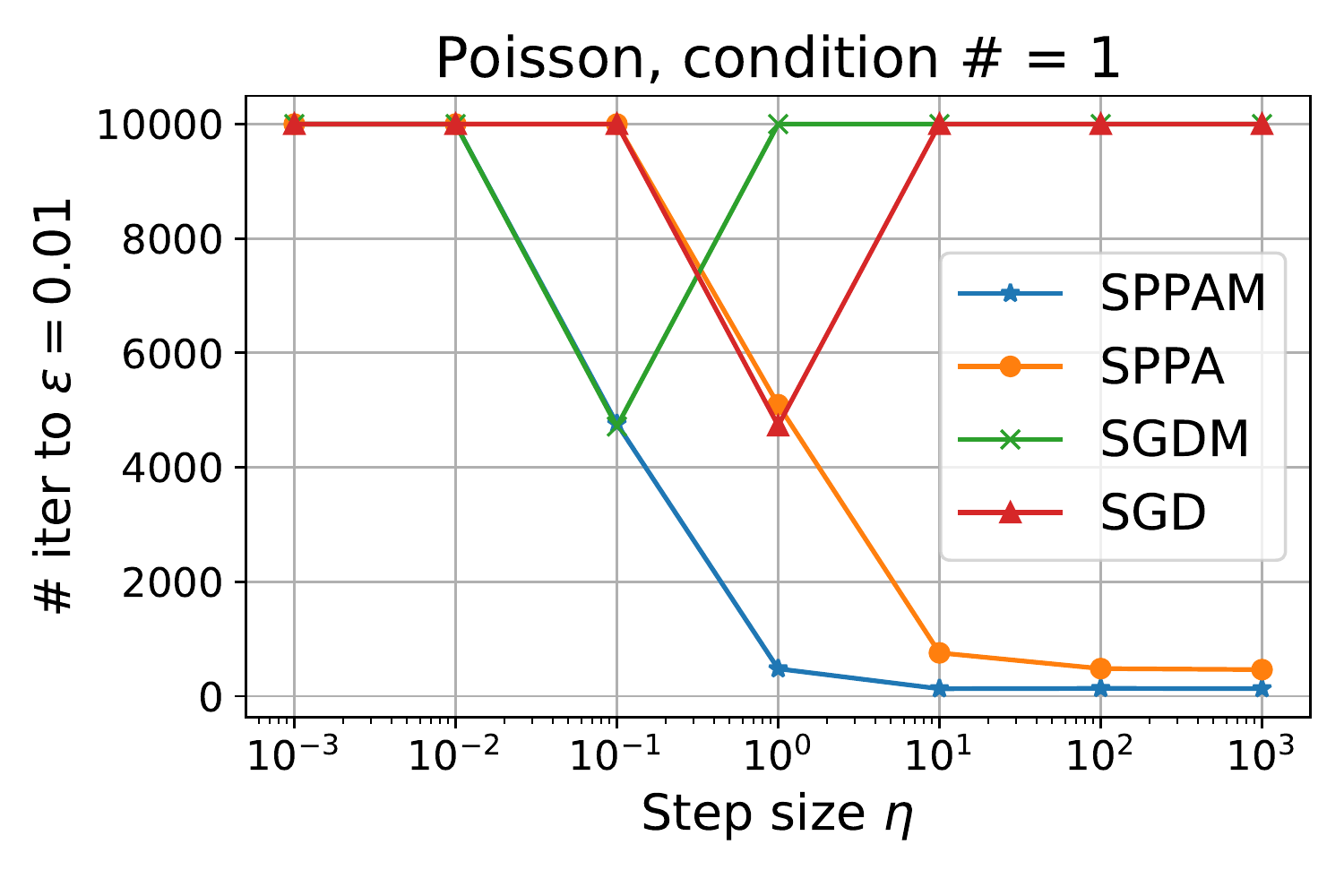}
    \includegraphics[scale=0.32]{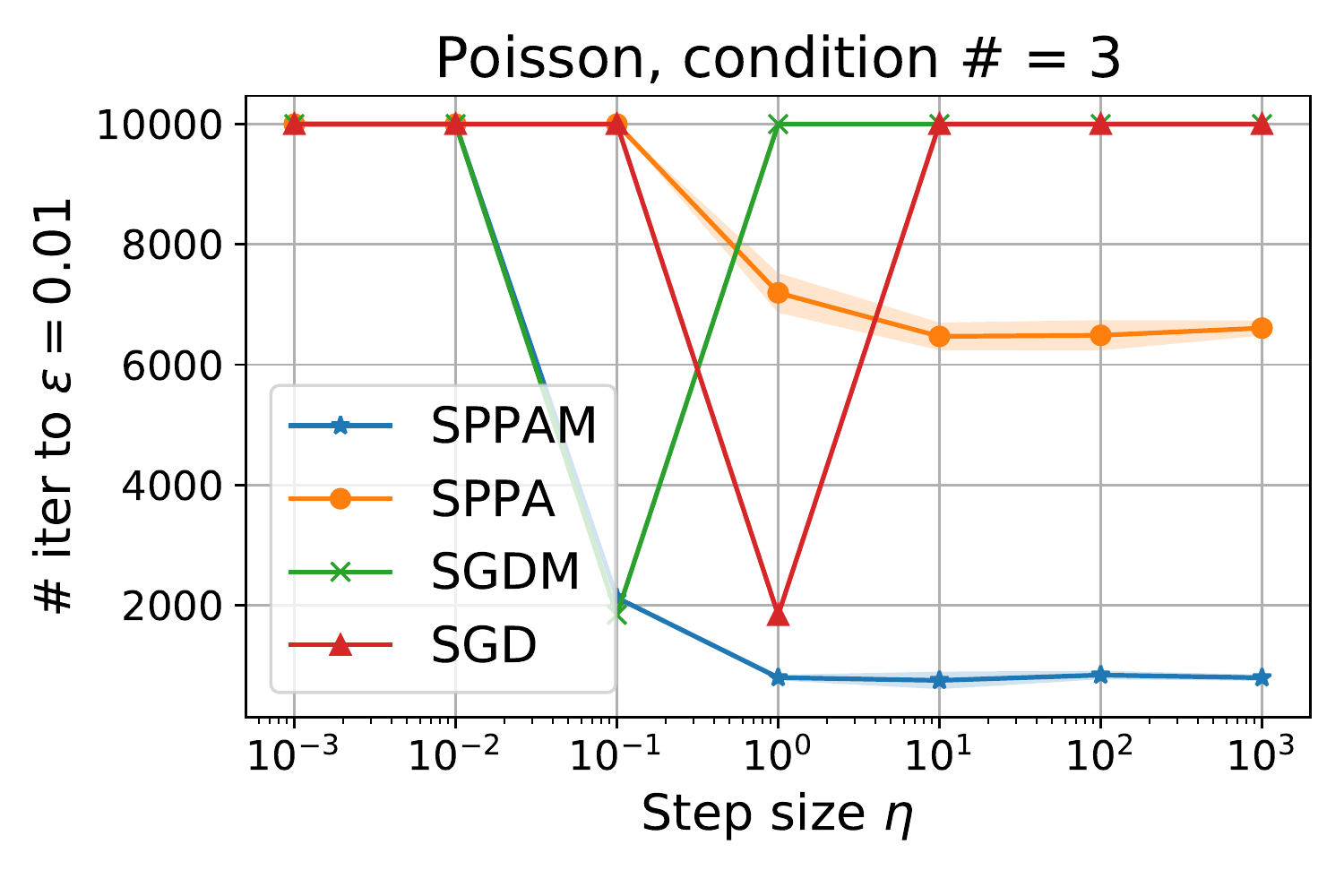}
    \includegraphics[scale=0.32]{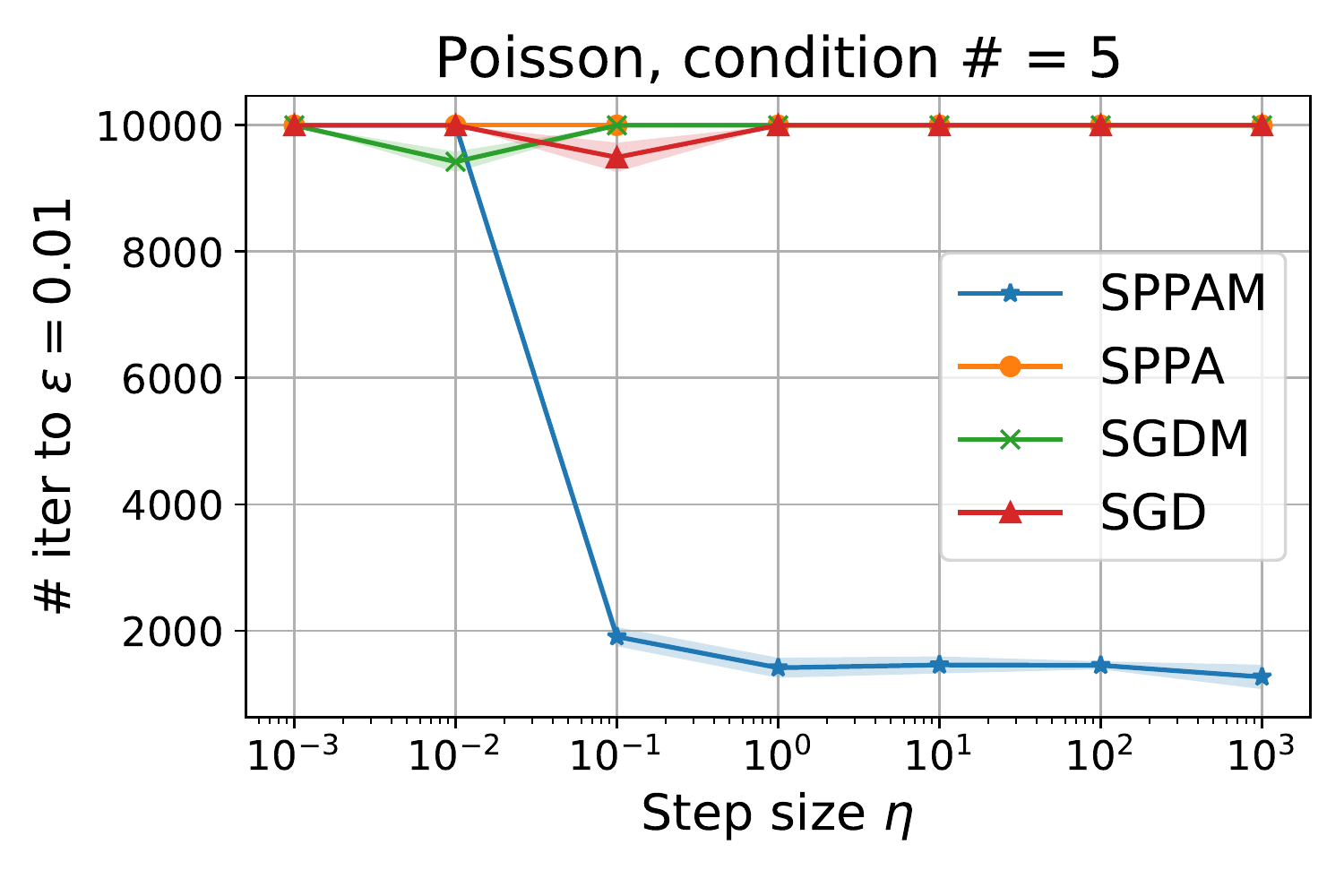}
    \caption{
    {\sf{Top}}: Linear regression with condition number $\kappa \in \{1, 5, 10\}$ with gaussian noise level $\texttt{1e-3}.$ 
    {\sf{Bottom}}: Poisson regression with condition number $\kappa \in \{1, 3, 5\}$.
    We set $p = n = 100$ in both cases. Batch size is 10 for all algorithms. The median number of iterations to reach $\varepsilon = 0.01$ is plotted. Shaded area are the standard deviations across 5 experiments.
    }
    \label{fig:regression}
\end{figure*} 
We generate the data as follows. $A \in \mathbb{R}^{p \times n}$ and $x^\star \in \mathbb{R}^p$ are drawn from $\mathcal{N}(0, I).$ For the normal case, we generate $b_i = \langle a_i, x^\star \rangle$, and for the Poisson case, we generate $b_i \sim \text{Poisson} (e^{\langle a_i, x^\star \rangle} )$ for $i = 1, \dots, n.$
For each experimental setup, we run SPPAM ({\sf \textcolor{blue}{blue}}), SPPA ({\sf \textcolor{orange}{orange}}), SGDM ({\sf \textcolor{teal}{green}}), and SGD ({\sf \textcolor{red}{red}}) for $10^4$ iterations. 
We repeat each experiment for 5 independent trials, and 
plot the median number of iterations to reach precision $\varepsilon \leq 10^{-2},$ along with the standard deviation.
We measure the precision 
$\varepsilon = \frac{\|b - \hat{b}\|_2^2}{\|b\|_2^2},$ 
where $b$ is the true label and $\hat{b}$ is the predicted label.

In Figure~\ref{fig:regression} ({\sf Top}), we present the results for the linear regression with different condition numbers, with gaussian noise level \texttt{1e-3}. We run each algorithm constant step size $\eta$ varying from $10^{-3}$ to $10^3$ with $10\times$ increment, and with $\beta=0.9$.
As expected, SGD and SGDM only converge for specific step size $\eta$, while SPPA and SPPAM converge for much wider ranges.
In terms of convergence rate, SPPAM converges faster than SPPA in all scenarios, which improves or matches the rate of SGDM, when it converges. 
As $\kappa$ increases, the range of $\eta$ that leads to convergence for SGD and SGDM shrinks; notice the sharper $``\lor"$ shape for SGD and SGDM for $\kappa = 10$ ({\sf 3rd}), compared to $\kappa=5$ ({\sf 2nd}) or $\kappa=1$ ({\sf 1st}).
SPPA also slightly slows down as $\kappa$ increases, while SPPAM converges essentially in the same manner for all scenarios.

Such trend is much more pronounced for the Poisson regression case presented in Figure~\ref{fig:regression} ({\sf Bottom}).
Due to the exponential mean function $h(\cdot)$ for Poisson model, the outcomes are extremely sensitive, and its likelihood does not satisfy standard assumptions like $L$-smoothness. As such, SGD and SGDM struggles with slow convergence even when $\kappa = 1$ ({\sf 1st}), while also exhibiting instability---each method converges only for a single choice of $\eta$ considered. 
Similar trend is shown when $\kappa=3$ ({\sf 2nd}) where SPPA starts slowing down.
For $\kappa = 5$ ({\sf 3rd}), all methods except for SPPAM did not make much progress in $10^4$ iterations, for the entire range of $\eta$ and $\beta$ considered. Quite remarkably, SPPAM still converges in the same manner without sacrificing both the convergence rate and the range of hyperparameters that lead to convergence.

\section{Conclusion}
We propose the stochastic proximal point algorithm with momentum (SPPAM), which directly incorporates Polyak's momentum inside the proximal step. We show that SPPAM converges to a neighborhood at a faster rate than stochastic proximal point algorithm (SPPA), and characterize the conditions that result in acceleration. Further, we prove linear convergence of SPPAM to a neighborhood, and provide conditions that lead to an exponential discount of the initial conditions, akin to SPPA. We confirm our theory with numerical simulations on linear and Poisson regression models; SPPAM converges for all the step sizes that SPPA converges, with a faster rate that matches or improves SGDM.

\acks{This work is supported by NSF FET:Small no.1907936, NSF MLWiNS CNS no.2003137 (in collaboration with Intel), NSF CMMI no.2037545, NSF CAREER award no.2145629, and Rice InterDisciplinary Excellence Award (IDEA).}

\bibliography{ref}

\begin{thebibliography}{51}
\providecommand{\natexlab}[1]{#1}
\providecommand{\url}[1]{\texttt{#1}}
\expandafter\ifx\csname urlstyle\endcsname\relax
  \providecommand{\doi}[1]{doi: #1}\else
  \providecommand{\doi}{doi: \begingroup \urlstyle{rm}\Url}\fi

\bibitem[Ahn and Sra(2022)]{ahn_proximal_2020}
Kwangjun Ahn and Suvrit Sra.
\newblock Understanding nesterov's acceleration via proximal point method.
\newblock In \emph{Symposium on Simplicity in Algorithms (SOSA)}, pages
  117--130. SIAM, 2022.

\bibitem[Allen-Zhu and Orecchia(2017)]{allen2017linear}
Z.~Allen-Zhu and L.~Orecchia.
\newblock Linear coupling: An ultimate unification of gradient and mirror
  descent.
\newblock In \emph{8th Innovations in Theoretical Computer Science Conference
  (ITCS 2017)}. Schloss Dagstuhl-Leibniz-Zentrum fuer Informatik, 2017.

\bibitem[Asi and Duchi(2019)]{asi_stochastic_2019}
Hilal Asi and John~C Duchi.
\newblock Stochastic (approximate) proximal point methods: Convergence,
  optimality, and adaptivity.
\newblock \emph{SIAM Journal on Optimization}, 29\penalty0 (3):\penalty0
  2257--2290, 2019.

\bibitem[Asi et~al.(2020)Asi, Chadha, and Cheng]{asi_minibatch_2020}
Hilal Asi, Karan Chadha, and Gary Cheng.
\newblock Minibatch {Stochastic} {Approximate} {Proximal} {Point} {Methods}.
\newblock \emph{34th Conference on Neural Information Processing Systems},
  page~11, 2020.

\bibitem[Assran and Rabbat(2020)]{assran_convergence_2020}
Mahmoud Assran and Michael Rabbat.
\newblock On the {Convergence} of {Nesterov}’s {Accelerated} {Gradient}
  {Method} in {Stochastic} {Settings}.
\newblock \emph{Proceedings of the 37 th International Conference on Machine
  Learning}, 2020.

\bibitem[Bach and Moulines(2013)]{bach2013non}
Francis Bach and Eric Moulines.
\newblock Non-strongly-convex smooth stochastic approximation with convergence
  rate $o(1/n)$.
\newblock In \emph{Proceedings of the 26th International Conference on Neural
  Information Processing Systems-Volume 1}, pages 773--781, 2013.

\bibitem[Bottou(2012)]{bottou_stochastic_2012}
L{\'e}on Bottou.
\newblock Stochastic gradient descent tricks.
\newblock In \emph{Neural networks: Tricks of the trade}, pages 421--436.
  Springer, 2012.

\bibitem[Bottou and Bousquet(2007)]{bottou2007tradeoffs}
L{\'e}on Bottou and Olivier Bousquet.
\newblock The tradeoffs of large scale learning.
\newblock \emph{Advances in neural information processing systems}, 20, 2007.

\bibitem[Bottou et~al.(2018)Bottou, Curtis, and
  Nocedal]{bottou_optimization_2018}
L{\'e}on Bottou, Frank~E Curtis, and Jorge Nocedal.
\newblock Optimization methods for large-scale machine learning.
\newblock \emph{Siam Review}, 60\penalty0 (2):\penalty0 223--311, 2018.

\bibitem[Bubeck et~al.(2015)Bubeck, Lee, and Singh]{bubeck2015geometric}
S.~Bubeck, Y.-T. Lee, and M.~Singh.
\newblock A geometric alternative to {N}esterov's accelerated gradient descent.
\newblock \emph{arXiv preprint arXiv:1506.08187}, 2015.

\bibitem[Chadha et~al.(2021)Chadha, Cheng, and Duchi]{chadha_accelerated_2021}
Karan Chadha, Gary Cheng, and John~C Duchi.
\newblock Accelerated, optimal, and parallel: Some results on model-based
  stochastic optimization.
\newblock \emph{arXiv preprint arXiv:2101.02696}, 2021.

\bibitem[d'Aspremont(2008)]{d2008smooth}
Alexandre d'Aspremont.
\newblock Smooth optimization with approximate gradient.
\newblock \emph{SIAM Journal on Optimization}, 19\penalty0 (3):\penalty0
  1171--1183, 2008.

\bibitem[Defazio(2019)]{defazio2019curved}
A.~Defazio.
\newblock On the curved geometry of accelerated optimization.
\newblock In \emph{Advances in Neural Information Processing Systems}, pages
  1764--1773, 2019.

\bibitem[Deng and Gao(2021)]{deng_minibatch_2021}
Qi~Deng and Wenzhi Gao.
\newblock Minibatch and momentum model-based methods for stochastic weakly
  convex optimization.
\newblock \emph{Advances in Neural Information Processing Systems}, 34, 2021.

\bibitem[Devolder et~al.(2014)Devolder, Glineur, and
  Nesterov]{devolder2014first}
Olivier Devolder, Fran{\c{c}}ois Glineur, and Yurii Nesterov.
\newblock First-order methods of smooth convex optimization with inexact
  oracle.
\newblock \emph{Mathematical Programming}, 146\penalty0 (1):\penalty0 37--75,
  2014.

\bibitem[Goh(2017)]{goh2017why}
Gabriel Goh.
\newblock Why momentum really works.
\newblock \emph{Distill}, 2\penalty0 (4):\penalty0 e6, 2017.

\bibitem[Gower et~al.(2019)Gower, Loizou, Qian, Sailanbayev, Shulgin, and
  Richtarik]{gower_sgd_2019}
Robert~M Gower, Nicolas Loizou, Xun Qian, Alibek Sailanbayev, Egor Shulgin, and
  Peter Richtarik.
\newblock {SGD}: {General} {Analysis} and {Improved} {Rates}.
\newblock \emph{Proceedings of the 36 th International Conference on Machine
  Learning}, page~10, 2019.

\bibitem[G{\"u}ler(1991)]{guler_convergence_1991}
Osman G{\"u}ler.
\newblock On the convergence of the proximal point algorithm for convex
  minimization.
\newblock \emph{SIAM journal on control and optimization}, 29\penalty0
  (2):\penalty0 403--419, 1991.

\bibitem[G{\"u}ler(1992)]{guler_new_1992}
Osman G{\"u}ler.
\newblock New proximal point algorithms for convex minimization.
\newblock \emph{SIAM Journal on Optimization}, 2\penalty0 (4):\penalty0
  649--664, 1992.

\bibitem[He et~al.(2016)He, Zhang, Ren, and Sun]{he2016deep}
Kaiming He, Xiangyu Zhang, Shaoqing Ren, and Jian Sun.
\newblock Deep residual learning for image recognition.
\newblock In \emph{Proceedings of the IEEE conference on computer vision and
  pattern recognition}, pages 770--778, 2016.

\bibitem[Howard et~al.(2017)Howard, Zhu, Chen, Kalenichenko, Wang, Weyand,
  Andreetto, and Adam]{howard2017mobilenets}
Andrew~G Howard, Menglong Zhu, Bo~Chen, Dmitry Kalenichenko, Weijun Wang,
  Tobias Weyand, Marco Andreetto, and Hartwig Adam.
\newblock Mobilenets: Efficient convolutional neural networks for mobile vision
  applications.
\newblock \emph{arXiv preprint arXiv:1704.04861}, 2017.

\bibitem[Hu and Lessard(2017)]{hu2017dissipativity}
B.~Hu and L.~Lessard.
\newblock Dissipativity theory for {N}esterov's accelerated method.
\newblock In \emph{Proceedings of the 34th International Conference on Machine
  Learning-Volume 70}, pages 1549--1557. JMLR. org, 2017.

\bibitem[Huang et~al.(2017)Huang, Liu, Van Der~Maaten, and
  Weinberger]{huang2017densely}
Gao Huang, Zhuang Liu, Laurens Van Der~Maaten, and Kilian~Q Weinberger.
\newblock Densely connected convolutional networks.
\newblock In \emph{Proceedings of the IEEE conference on computer vision and
  pattern recognition}, pages 4700--4708, 2017.

\bibitem[Kidambi et~al.(2018)Kidambi, Netrapalli, Jain, and
  Kakade]{kidambi_insufficiency_2018}
Rahul Kidambi, Praneeth Netrapalli, Prateek Jain, and Sham Kakade.
\newblock On the insufficiency of existing momentum schemes for stochastic
  optimization.
\newblock In \emph{2018 Information Theory and Applications Workshop (ITA)},
  pages 1--9. IEEE, 2018.

\bibitem[Kim et~al.(2021)Kim, Toulis, and Kyrillidis]{kim2021convergence}
Junhyung~Lyle Kim, Panos Toulis, and Anastasios Kyrillidis.
\newblock Convergence and stability of the stochastic proximal point algorithm
  with momentum.
\newblock \emph{arXiv preprint arXiv:2111.06171}, 2021.

\bibitem[Kulunchakov and Mairal(2019)]{kulunchakov_generic_2019}
Andrei Kulunchakov and Julien Mairal.
\newblock A {Generic} {Acceleration} {Framework} for {Stochastic} {Composite}
  {Optimization}.
\newblock \emph{Advances in Neural Information Processing Systems}, 32, October
  2019.
\newblock arXiv: 1906.01164.

\bibitem[Laborde and Oberman(2020)]{laborde2019lyapunov}
Maxime Laborde and Adam Oberman.
\newblock A lyapunov analysis for accelerated gradient methods: from
  deterministic to stochastic case.
\newblock In \emph{International Conference on Artificial Intelligence and
  Statistics}, pages 602--612. PMLR, 2020.

\bibitem[Lessard et~al.(2016)Lessard, Recht, and Packard]{lessard2016analysis}
L.~Lessard, B.~Recht, and A.~Packard.
\newblock Analysis and design of optimization algorithms via integral quadratic
  constraints.
\newblock \emph{SIAM Journal on Optimization}, 26\penalty0 (1):\penalty0
  57--95, 2016.

\bibitem[Lin et~al.(2018)Lin, Mairal, and Harchaoui]{lin_catalyst_2018}
H.~Lin, Julien Mairal, and Zaid Harchaoui.
\newblock Catalyst {Acceleration} for {First}-order {Convex} {Optimization}:
  from {Theory} to {Practice}.
\newblock \emph{Journal of Machine Learning Research}, 18:\penalty0 1--54,
  2018.

\bibitem[Lin et~al.(2015)Lin, Mairal, and Harchaoui]{lin_universal_2015}
Hongzhou Lin, Julien Mairal, and Zaid Harchaoui.
\newblock A universal catalyst for first-order optimization.
\newblock \emph{Advances in neural information processing systems}, 28, 2015.

\bibitem[Liu and Belkin(2018)]{liu_accelerating_2019}
Chaoyue Liu and Mikhail Belkin.
\newblock Accelerating sgd with momentum for over-parameterized learning.
\newblock \emph{arXiv preprint arXiv:1810.13395}, 2018.

\bibitem[Loizou and Richt{\'a}rik(2020)]{loizou2020momentum}
Nicolas Loizou and Peter Richt{\'a}rik.
\newblock Momentum and stochastic momentum for stochastic gradient, newton,
  proximal point and subspace descent methods.
\newblock \emph{Computational Optimization and Applications}, 77\penalty0
  (3):\penalty0 653--710, 2020.

\bibitem[Moulines and Bach(2011)]{moulines_non-asymptotic_2011}
Eric Moulines and Francis~R. Bach.
\newblock Non-{Asymptotic} {Analysis} of {Stochastic} {Approximation}
  {Algorithms} for {Machine} {Learning}.
\newblock In J.~Shawe-Taylor, R.~S. Zemel, P.~L. Bartlett, F.~Pereira, and
  K.~Q. Weinberger, editors, \emph{Advances in {Neural} {Information}
  {Processing} {Systems} 24}, pages 451--459. Curran Associates, Inc., 2011.

\bibitem[Nelder and Wedderburn(1972)]{nelder1972generalized}
John~Ashworth Nelder and Robert~WM Wedderburn.
\newblock Generalized linear models.
\newblock \emph{Journal of the Royal Statistical Society: Series A (General)},
  135\penalty0 (3):\penalty0 370--384, 1972.

\bibitem[Nemirovski et~al.(2009)Nemirovski, Juditsky, Lan, and
  Shapiro]{nemirovski_robust_2009}
Arkadi Nemirovski, Anatoli Juditsky, Guanghui Lan, and Alexander Shapiro.
\newblock Robust stochastic approximation approach to stochastic programming.
\newblock \emph{SIAM Journal on optimization}, 19\penalty0 (4):\penalty0
  1574--1609, 2009.

\bibitem[Nesterov et~al.(2018)]{nesterov_lectures_2018}
Yurii Nesterov et~al.
\newblock \emph{Lectures on convex optimization}, volume 137.
\newblock Springer, 2018.

\bibitem[Nesterov(1983)]{nesterov1983method}
Yurii~E Nesterov.
\newblock A method for solving the convex programming problem with convergence
  rate $o(1/k^2)$.
\newblock In \emph{Dokl. akad. nauk Sssr}, volume 269, pages 543--547, 1983.

\bibitem[Polyak(1964)]{polyak_methods_1964}
Boris~T Polyak.
\newblock Some methods of speeding up the convergence of iteration methods.
\newblock \emph{Ussr computational mathematics and mathematical physics},
  4\penalty0 (5):\penalty0 1--17, 1964.

\bibitem[Polyak(1987)]{polyak1987introduction}
Boris~T. Polyak.
\newblock Introduction to optimization.
\newblock \emph{Inc., Publications Division, New York}, 1, 1987.

\bibitem[Polyak and Juditsky(1992)]{polyak1992acceleration}
Boris~T Polyak and Anatoli~B Juditsky.
\newblock Acceleration of stochastic approximation by averaging.
\newblock \emph{SIAM journal on control and optimization}, 30\penalty0
  (4):\penalty0 838--855, 1992.

\bibitem[Robbins and Monro(1951)]{robbins_stochastic_1951}
Herbert Robbins and Sutton Monro.
\newblock A stochastic approximation method.
\newblock \emph{The annals of mathematical statistics}, pages 400--407, 1951.

\bibitem[Rockafellar(1976)]{rockafellar_monotone_1976}
R~Tyrrell Rockafellar.
\newblock Monotone operators and the proximal point algorithm.
\newblock \emph{SIAM journal on control and optimization}, 14\penalty0
  (5):\penalty0 877--898, 1976.

\bibitem[Ryu and Boyd(2017)]{ryu_stochastic_2017}
Ernest~K Ryu and Stephen Boyd.
\newblock Stochastic {Proximal} {Iteration}: {A} {Non}-{Asymptotic}
  {Improvement} {Upon} {Stochastic} {Gradient} {Descent}.
\newblock \emph{Author website}, page~42, 2017.

\bibitem[Shalev-Shwartz et~al.(2011)Shalev-Shwartz, Singer, Srebro, and
  Cotter]{shalev2011pegasos}
Shai Shalev-Shwartz, Yoram Singer, Nathan Srebro, and Andrew Cotter.
\newblock Pegasos: Primal estimated sub-gradient solver for {SVM}.
\newblock \emph{Mathematical programming}, 127\penalty0 (1):\penalty0 3--30,
  2011.

\bibitem[Su et~al.(2014)Su, Boyd, and Candes]{su2014differential}
W.~Su, S.~Boyd, and E.~Candes.
\newblock A differential equation for modeling {N}esterov’s accelerated
  gradient method: Theory and insights.
\newblock In \emph{Advances in Neural Information Processing Systems}, pages
  2510--2518, 2014.

\bibitem[Toulis et~al.(2014)Toulis, Airoldi, and
  Rennie]{toulis_statistical_2014}
Panagiotis Toulis, Edoardo Airoldi, and Jason Rennie.
\newblock Statistical analysis of stochastic gradient methods for generalized
  linear models.
\newblock In \emph{International Conference on Machine Learning}, pages
  667--675. PMLR, 2014.

\bibitem[Toulis and Airoldi(2017)]{toulis_asymptotic_2017}
Panos Toulis and Edoardo~M Airoldi.
\newblock Asymptotic and finite-sample properties of estimators based on
  stochastic gradients.
\newblock \emph{The Annals of Statistics}, 45\penalty0 (4):\penalty0
  1694--1727, 2017.

\bibitem[Toulis et~al.(2021)Toulis, Horel, and Airoldi]{toulis_proximal_2021}
Panos Toulis, Thibaut Horel, and Edoardo~M Airoldi.
\newblock The proximal robbins--monro method.
\newblock \emph{Journal of the Royal Statistical Society: Series B (Statistical
  Methodology)}, 83\penalty0 (1):\penalty0 188--212, 2021.

\bibitem[Wibisono et~al.(2016)Wibisono, Wilson, and
  Jordan]{wibisono2016variational}
A.~Wibisono, A.~Wilson, and M.~Jordan.
\newblock A variational perspective on accelerated methods in optimization.
\newblock \emph{proceedings of the National Academy of Sciences}, 113\penalty0
  (47):\penalty0 E7351--E7358, 2016.

\bibitem[Williams(1992)]{williams1992nth}
K.~Williams.
\newblock The $n$-th power of a 2$\times$ 2 matrix.
\newblock \emph{Mathematics Magazine}, 65\penalty0 (5):\penalty0 336--336,
  1992.

\bibitem[Zhang(2004)]{zhang_2004_solving}
Tong Zhang.
\newblock Solving large scale linear prediction problems using stochastic
  gradient descent algorithms.
\newblock In \emph{Proceedings of the twenty-first international conference on
  Machine learning}, page 116, 2004.

\end{thebibliography}

\newpage
\section*{Supplementary Material}

\appendix


\section{Proofs for Section \ref{sec:quad-model}}

In this section, we provide proofs for the propositions in Section \ref{sec:quad-model}. Proofs below utilize ``classical momentum'' form, which iterates:
\begin{align*}
    y_{t+1} &= \beta y_t + \nabla f(x_t) \\ 
    x_{t+1} &= x_t - \eta y_{t+1}.
\end{align*}
 This is equivalent to Polyak's momentum in the sense that, plugging in the first equation to the second, we get
 \begin{align*}
    x_{t+1} &= x_t - \eta y_{t+1} = x_t - \eta \nabla f(x_t) - \eta \beta y_t \\
    &= x_t - \eta \nabla f(x_t) + \beta (x_t - x_{t-1}),
\end{align*}
where the last equality is from the second equation of classical momentum.

\subsection*{Proof of Proposition \ref{prop:ppa}}

\begin{proof}
Recall PPA/IGD recursion. For quadratic problem in \eqref{eq:obj-quad}, the gradient can be computed in closed form:
\begin{align*}
    x_{t+1} &= x_t - \eta \nabla f(x_{t+1}) = x_t - \eta (Ax_{t+1} - b).
\end{align*}
Consider the eigenvalue decomposition of $A = QDQ^\top$ and the change of basis $z_t = Q^\top (x_t  - x^\star).$ Then, $Q z_t = x_t - x^\star$ and $Q z_t + x^\star = x_t.$ Also using $AQ = QD$ and $Ax^\star = b,$ we can write down the recursion above as:
\begin{align*}
    Q z_{t+1} + x^\star &= Q z_t + x^\star - \eta \left( A(Q z_{t+1} + x^\star) - b \right) \\
    &= Qz_t + x^\star - \eta Q D z_{t+1}  
\end{align*}

Multiplying $Q^\top$ on both sides,
\begin{align*}
    z_{t+1} + Q^\top x^\star &= z_t + Q^\top x^\star - \eta D z_{t+1} \Rightarrow \\
    (1+\eta D) z_{t+1} &= z_t.
\end{align*}

Writing the above component-wise, we get
\begin{align*}
    z_{t+1}^i &= \left( \frac{1}{1 + \eta \lambda_i}\right) \cdot z_t^i = \left( \frac{1}{1 + \eta \lambda_i} \right)^{t+1} \cdot z_0^i
\end{align*}
Going back to the change of basis, $z_t = Q^\top (x_t  - x^\star),$ we have the relation $x_t  - x^\star = Q z_t.$ Therefore, using the component-wise relation for $z_{t+1}^i$ above,
\begin{align*}
    x_t  - x^\star = Q z_t = \sum_{i=1}^n z_0^i \left( \frac{1}{1 + \eta \lambda_i} \right)^t q^i.
\end{align*}
Thus, in order for IGD to converge, one needs to satisfy $\left| \frac{1}{1+\eta \lambda_i} \right| < 1.$ 
\end{proof}

\subsection*{Proof of Proposition \ref{prop:ppam}}

\begin{proof}
Consider the ``classical momentum" form of PPAM:
\begin{equation}
\begin{split}
\label{eq:iGD-CM}
  y^{k+1} &= \beta y^k + \nabla f(x^{k+1}) \\
  x^{k+1} &= x^k - \eta y^{k+1}.  
\end{split}
\end{equation}
We perform change of basis: $z^k = Q^\top (x^k - x^\star)$ and $\phi^k = Q^\top y^k$.

For the first line of \eqref{eq:iGD-CM}, we have:
\begin{align*}
  y^{k+1} &= \beta y^k + \nabla f(x^{k+1}) \\
  Q^\top y^{k+1} &= \beta Q^\top y^k + Q^\top(A x^{k+1} - b) \\
  &= \beta Q^\top y^k + Q^\top (A (Qz^{k+1} + x^\star) - b) \\
  &= \beta Q^\top y^k + Q^\top (AQz^{k+1} + Ax^\star - b) \\
  &= \beta Q^\top y^k + Q^\top AQz^{k+1} \\
  &= \beta Q^\top y^k + DQ^\top Q z^{k+1} \\
  &= \beta Q^\top y^k + D z^{k+1} .
\end{align*}
Change of basis and writing component-wise, we get:
\begin{align*}
  \phi_i^{k+1} &= \beta \phi_i^k + \lambda_i z_i^{k+1} \\
  &= \beta \phi_i^k + \lambda_i (z_i^k - \eta \phi_i^{k+1}) \\
  &= \beta \phi_i^k + \lambda_i z_i^k - \eta \lambda_i  \phi_i^{k+1} \\
  (1+\eta \lambda_i ) \phi_i^{k+1} &= \beta \phi_i^k + \lambda_i z_i^k \\
  \phi_i^{k+1} &= \frac{\beta}{1+\eta \lambda_i }\phi_i^k + \frac{\lambda_i}{1+\eta \lambda_i } z_i^k.
\end{align*}

For the second line of \eqref{eq:iGD-CM}, we have:
\begin{align*}
  x^{k+1} &= x^k - \eta y^{k+1} \\
  Qz^{k+1} + x^\star &= Q z^k + x^\star - \eta Q \phi^{k+1} \\
  Q^\top Q z^{k+1} &= Q^\top Q z^k - \eta Q^\top Q \phi^{k+1} .
\end{align*}
Again, change of basis and writing component-wise, we get:
\begin{align*}
  z_i^{k+1} &= z_i^k - \eta \phi_i^{k+1}.
\end{align*}
Therefore, \eqref{eq:iGD-CM} can be written as, component-wise,
\begin{align*}
  \phi_i^{k+1} &= \frac{\beta}{1+\eta \lambda_i } \phi_i^k + \frac{\lambda_i}{1+\eta \lambda_i } z_i^k \\
  z_i^{k+1} &= z_i^k - \eta \phi_i^{k+1}.
\end{align*}
We can write above in matrix form:
\begin{align*}
\begin{bmatrix}
  1 & 0  \\
  \eta & 1
\end{bmatrix} \cdot
\begin{bmatrix}
  \phi_i^{k+1} \\
  z_i^{k+1}
\end{bmatrix} &=
\begin{bmatrix}
  \frac{\beta}{1+\eta \lambda_i } & \frac{\lambda_i}{1+\eta \lambda_i } \\
  0 & 1  
\end{bmatrix} \cdot
\begin{bmatrix}
  \phi_i^k \\
  z_i^k
\end{bmatrix}.
\end{align*}
Multiplying the inverse of the first matrix, i.e., $\begin{bmatrix}
  1 & 0  \\
  \eta & 1
\end{bmatrix}^{-1} =
\begin{bmatrix}
  1 & 0  \\
  -\eta & 1
\end{bmatrix}$ 
on both sides,
\begin{align*}
\begin{bmatrix}
  \phi_i^{k+1} \\
  z_i^{k+1}
\end{bmatrix} &=
\begin{bmatrix}
  1 & 0  \\
  -\eta & 1
\end{bmatrix} \cdot
\begin{bmatrix}
  \frac{\beta}{1+\eta \lambda_i } & \frac{\lambda_i}{1+\eta \lambda_i } \\
  0 & 1  
\end{bmatrix} \cdot
\begin{bmatrix}
  \phi_i^k \\
  z_i^k
\end{bmatrix} \\ &=
\begin{bmatrix}
  \frac{\beta}{1+\eta \lambda_i } & \frac{\lambda_i}{1+\eta \lambda_i } \\
  \frac{-\eta \beta}{1+\eta \lambda_i } & \frac{1}{1+\eta \lambda_i }
\end{bmatrix} \cdot
\begin{bmatrix}
  \phi_i^k \\
  z_i^k
\end{bmatrix}.
\end{align*}
Therefore, we can write the above as 
\begin{align*}
\begin{bmatrix}
  \phi_i^k \\
  z_i^k
\end{bmatrix} &=
R^k \cdot
\begin{bmatrix}
  \phi_i^0 \\
  z_i^0
\end{bmatrix}, \quad
R = \begin{bmatrix}
  \frac{\beta}{1+\eta \lambda_i} & \frac{\lambda_i}{1+\eta \lambda_i} \\
  \frac{-\eta \beta}{1+\eta \lambda_i} & \frac{1}{1+\eta \lambda_i}
\end{bmatrix}.
\end{align*}
To compute $R^k,$ we use the method presented in \cite{williams1992nth}. Then, 
denoting $\sigma_1$ and $\sigma_2$ as the eigenvalues of $R$, we have:
\begin{align*}
    R^k = 
    \begin{cases}
    \sigma_1^k R_1 - \sigma_2^k R_2  &\sigma_1 \neq \sigma_2 \\
    \sigma_1^k \left(k \frac{R}{\sigma_1} - (k-1)I \right) \quad &\sigma_1 = \sigma_2
    \end{cases},
    \quad
    R_j = \frac{R - \sigma_jI}{\sigma_1 - \sigma_2}.
\end{align*}
To get the convergence condition, we compute the eigenvalues of $R$ explicitly:
\begin{align*}
    \sigma_{1, 2} = \frac{1}{2} \left( \frac{\beta+1}{1+\eta \lambda_i} \pm \sqrt{ \left( \frac{\beta+1}{1+\eta \lambda_i} \right)^2 - \frac{4\beta}{1+\eta \lambda_i}} \right) .
\end{align*}

Now, to get the convergence criterion, we need to examine the conditions that lead to $\max \{ |\sigma_1|, |\sigma_2| \} < 1. $ If the eigenvalues computed above are complex, i.e.,
\begin{align*}
    \left( \frac{\beta+1}{1+\eta \lambda_i} \right)^2 - \frac{4\beta}{1+\eta \lambda_i} < 0 \implies |\sigma_1| = |\sigma_2| &= \sqrt{ \frac{1}{4}\left( \frac{\beta+1}{1+\eta \lambda_i} \right)^2 + \left\lvert \frac{1}{4} \left( \frac{\beta+1}{1+\eta \lambda_i} \right)^2 - \frac{\beta}{1+\eta \lambda_i} \right\rvert} \\
    & = \sqrt{\frac{\beta}{1+\eta \lambda_i}}.
\end{align*}
We need the above quantity to be less than $1$ to converge, so we need 
\begin{align*}
    \sqrt{\frac{\beta}{1+\eta \lambda_i}} < 1  \iff \beta - 1 < \eta \lambda_i \iff \eta > \frac{\beta - 1}{\lambda_i}.
\end{align*}

Now, if the eigenvalues are real, i.e.,
\begin{align*}
\left( \frac{\beta+1}{1+\eta \lambda_i} \right)^2 - \frac{4\beta}{1+\eta \lambda_i} \geq 0 \implies
    \max \{ |\sigma_1|, |\sigma_2| \} = \frac{1}{2} \max \left\{ \left\lvert \frac{\beta+1}{1+\eta \lambda_i} \pm \sqrt{ \left( \frac{\beta+1}{1+\eta \lambda_i} \right)^2 - \frac{4\beta}{1+\eta \lambda_i}} \right\rvert \right\}.
\end{align*}
Cases can be further divided into two. In the first case, when we have $\frac{\beta+1}{1+\eta \lambda_i} > 0$, we have $\sigma_1 \geq \sigma_2 \geq 0$, because the square-root term is non-negative. Therefore, to have $\max \{ |\sigma_1|, |\sigma_2| \} < 1,$  we need
\begin{align*}
    \sigma_1 =  \frac{1}{2} \left(  \frac{\beta+1}{1+\eta \lambda_i} + \sqrt{ \left( \frac{\beta+1}{1+\eta \lambda_i} \right)^2 - \frac{4\beta}{1+\eta \lambda_i}} \right) &< 1.
\end{align*}

In the second case, when we have $\frac{\beta+1}{1+\eta \lambda_i} < 0$, we have $|\sigma_2| \geq |\sigma_1|$. Therefore, to have $\max \{ |\sigma_1|, |\sigma_2| \} < 1,$  we need
\begin{align*}
    \sigma_2 =  \frac{1}{2} \left(  \frac{\beta+1}{1+\eta \lambda_i} - \sqrt{ \left( \frac{\beta+1}{1+\eta \lambda_i} \right)^2 - \frac{4\beta}{1+\eta \lambda_i}} \right) &> -1.
\end{align*}

\end{proof}

\section{Proofs for Section \ref{sec:theory}}

Recall the recursion of SPPAM:
\begin{align}
    x_{t+1}^+ 
    &= x_t - \eta \nabla f(x_{t+1}^+) + \beta (x_t - x_{t-1})  \label{eq:acc-ppa-supp} \\
    x_{t+1} &= x_{t+1}^+ - \eta  \varepsilon_{t+1}. \label{eq:acc-stoc-ppa-supp}
\end{align}
We will refer to \eqref{eq:acc-ppa-supp} as PPAM (without stochastic errors). 

\begin{lemma} \label{lemma:pl-condition}
For $\mu$-strongly convex $f(\cdot),$ it holds that for all $x$,
\begin{align*}
    \| \nabla f(x) \|_2^2 \geq \mu^2 \|x - x^\star \|_2^2.
\end{align*}
\end{lemma}

\begin{proof}
By strong convexity, we have for all $x$ and $y$, 
\begin{align*}
    f(y) \geq f(x) + \nabla f(x)^\top (y-x) + \tfrac{\mu}{2} \|y-x\|_2^2.
\end{align*}
Since minimization retains inequality, we can minimize each side. On the left hand side, we have $\min_y \{ f(y) \} = f(x^\star).$ On the right hand side, we take the  derivative with respect to $y$ and set it to zero to obtain: 
\begin{align*}
    \nabla f(x) + \mu (y-x) = 0 \Rightarrow y = x - \tfrac{1}{\mu} \nabla f(x).
\end{align*}
Plugging back, we get:
\begin{align*}
    f(x^\star) &\geq f(x) + \nabla f(x)^\top (x - \tfrac{1}{\mu} \nabla f(x) - x) + \tfrac{\mu}{2} \|x - x + \tfrac{1}{\mu} \nabla f(x) \|_2^2 \\
    &=f(x)  - \tfrac{1}{2\mu} \| \nabla f(x) \|_2^2.
\end{align*}
Rearraging, we have
\begin{align*}
     \| \nabla f(x) \|_2^2 &\geq 2 \mu(f(x) - f(x^\star)) \\
     &\geq \mu^2 \|x - x^\star \|_2^2,
\end{align*}
where last inequality uses strong convexity with $y = x$ and $x=x^\star.$
\end{proof}

\subsection*{Proof of Theorem \ref{thm:onestep-acc-stoc-prox}}

\begin{proof}
From PPAM in \eqref{eq:acc-ppa-supp}, subtract $x^\star$ on both sides and take the norm squared:
\begin{align*}
    x_{t+1}^+ - x^\star &= x_t - x^\star - \eta \nabla f(x_{t+1}^+) + \beta (x_t - x_{t-1}) \Rightarrow\\
    \|x_{t+1}^+ - x^\star\|_2^2 &= \| x_t - x^\star \|_2^2 + \eta^2 \| \nabla f(x_{t+1}^+) \|_2^2 + \beta^2 \| x_t - x_{t-1}\|_2^2 \\ 
    & \quad - 2 \eta (x_t - x^\star)^\top \nabla f(x_{t+1}^+) \quad  \\
    & \quad + 2\beta(x_t-x^\star)^\top (x_t-x_{t-1}) \quad  \\
    & \quad - 2 \beta \eta (x_t - x_{t-1})^\top \nabla f(x_{t+1}^+).
\end{align*}
For the fourth term, observe that
\begin{align*}
    (x_t - x^\star)^\top \nabla f(x_{t+1}^+) &= (x_{t+1}^+ - x^\star + \eta \nabla f(x_{t+1}^+) - \beta(x_t - x_{t-1}))^\top \nabla f(x_{t+1}^+) \\
    &= (x_{t+1}^+ - x^\star)^\top \nabla f(x_{t+1}^+) + \eta \|\nabla f(x_{t+1}^+) \|_2^2 - \beta(x_t - x_{t-1})^\top \nabla f(x_{t+1}^+) \Rightarrow \\
    -2\eta(x_t-x^\star)^\top \nabla f(x_{t+1}^+) &= -2\eta (x_{t+1}^+ - x^\star)^\top \nabla f(x_{t+1}^+) -2\eta^2 \|\nabla f(x_{t+1}^+) \|_2^2 + 2  \beta \eta(x_t - x_{t-1})^\top \nabla f(x_{t+1}^+) \\
    &\leq -2\eta \mu \|x_{t+1}^+ - x^\star\|_2^2 -2\eta^2 \|\nabla f(x_{t+1}^+) \|_2^2 + 2 \beta \eta(x_t - x_{t-1})^\top \nabla f(x_{t+1}^+),
\end{align*}
where the last inequality uses strong convexity of $f(\cdot)$.

For the third and fifth term, observe that
\begin{align*}
    \beta^2 \| x_t - x_{t-1}\|_2^2 + 2\beta (x_t - x^\star)^\top (x_t - x_{t-1})
    \hspace{-4cm}\\
    &= 
    \|\beta(x_t - x_{t-1}) + (x_t -x^\star) \|_2^2 - \|x_t - x^\star \|_2^2  \\
    &= \|(1+\beta) x_t - \beta x_{t-1} - x^\star \|_2^2 - \|x_t - x^\star \|_2^2 \nonumber \\
    &= \|(1+\beta)(x_t - x^\star) - \beta (x_{t-1} -x^\star) \|_2^2 - \|x_t - x^\star \|_2^2  \\
    &\leq (1+\beta)^2(1+\zeta) \|x_t - x^\star\|_2^2 + \beta^2(1+\tfrac{1}{\zeta}) \|x_{t-1} - x^\star \|_2^2 -  \|x_t - x^\star \|_2^2,
\end{align*}
where last inequality uses Young's inequality: $\|a+b\|_2^2 \leq (1+\zeta)\|a\|_2^2 + (1+\tfrac{1}{\zeta}) \|b\|_2^2$ for $\zeta > 0$.
Combining terms, we have
\begin{align*}
    \|x_{t+1}^+ - x^\star\|_2^2   &\leq  \| x_t - x^\star \|_2^2 + \eta^2 \| \nabla f(x_{t+1}^+) \|_2^2  \nonumber \\
    & \qquad -2\eta \mu \|x_{t+1}^+ - x^\star\|_2^2 -2\eta^2 \|\nabla f(x_{t+1}^+) \|_2^2 + 2  \cancel{\beta \eta(x_t - x_{t-1})^\top \nabla f(x_{t+1}^+)}  \nonumber \\
    & \qquad + (1+\beta)^2(1+\zeta)\|x_t - x^\star\|_2^2 + \beta^2(1+\tfrac{1}{\zeta}) \|x_{t-1} - x^\star \|_2^2 - \|x_t - x^\star \|_2^2  \nonumber \\
    &\qquad - 2 \cancel{\beta \eta (x_t - x_{t-1})^\top \nabla f(x_{t+1}^+)} \\
    & = -2\eta \mu \|x_{t+1}^+ - x^\star\|_2^2 -\eta^2 \|\nabla f(x_{t+1}^+) \|_2^2+ (1+\beta)^2(1+\zeta)\|x_t - x^\star\|_2^2 \\ 
    &\qquad + \beta^2(1+\tfrac{1}{\zeta}) \|x_{t-1} - x^\star \|_2^2 \\
    &\leq - (2 \eta \mu + \eta^2 \mu^2) \|x_{t+1}^+ - x^\star\|_2^2 + (1+\beta)^2(1+\zeta)\|x_t - x^\star\|_2^2 \\ 
    &\qquad + \beta^2(1+\tfrac{1}{\zeta}) \|x_{t-1} - x^\star \|_2^2,
\end{align*}
where the last inequality is by Lemma~\ref{lemma:pl-condition}.
Grouping the same terms, we get:
\begin{align}
    (1+\eta \mu)^2 \|x_{t+1}^+ - x^\star\|_2^2 &\leq (1+\beta)^2(1+\zeta)\|x_t - x^\star\|_2^2 + \beta^2(1+\tfrac{1}{\zeta}) \|x_{t-1} - x^\star \|_2^2 \Rightarrow \nonumber \\
    \|x_{t+1}^+ - x^\star\|_2^2 &\leq  \frac{(1+\beta)^2(1+\zeta)}{(1+\eta \mu)^2} \|x_t - x^\star\|_2^2 + \frac{\beta^2(1+\tfrac{1}{\zeta})}{(1+\eta \mu)^2} \|x_{t-1} - x^\star \|_2^2. \label{eq:main-recursion}
\end{align}
Now, choose $\zeta = \frac{4}{(1+\beta)^2}-1$, which is positive for $0 < \beta < 1$. Then, each coefficient in the two terms on the RHS above reduces to:
\begin{align*}
    \frac{(1+\beta)^2(1+\zeta)}{(1+\eta \mu)^2} =  \frac{4}{(1+\eta \mu)^2}, \quad\text{and}\quad  \frac{\beta^2(1+\frac{1}{\zeta})}{(1+\eta \mu)^2} = \frac{4\beta^2}{(1+\eta \mu)^2(4 - (1+\beta)^2)}.
\end{align*}
Therefore, our original recursion in \eqref{eq:main-recursion} reduces to
\begin{align}
\label{eq:igdhb-onestep}
    \|x_{t+1}^+ - x^\star\|_2^2 &\leq  \frac{4}{(1+\eta \mu)^2} \|x_t - x^\star\|_2^2 +
    \frac{4\beta^2}{(1+\eta \mu)^2(4 - (1+\beta)^2)} \|x_{t-1} - x^\star \|_2^2.
\end{align}

Note that from \eqref{eq:acc-ppa-supp}, we have 
\begin{align*}
    x_{t+1}^+ +  \eta \nabla f(x_{t+1}^+) = x_t + \beta (x_t - x_{t-1}).
\end{align*}
Thus, $x_{t+1}^+$ is deterministic given $x_t$ and $x_{t-1}.$
Therefore, going back to SPPAM in \eqref{eq:acc-stoc-ppa-supp} and taking expectations, we have:
\begin{align*}
    \ex{ \|x_{t+1 }- x^\star\|_2^2 } &=  \ex{ \| x_{t+1}^+ - x^\star \|_2^2 } - 2\eta \ex{ \langle x_{t+1}^+ - x^\star, \varepsilon_{t+1} \rangle } + \eta^2 \ex{\| \varepsilon_{t+1} \|_2^2} \\
    &= \ex{ \| x_{t+1}^+ - x^\star \|_2^2 } + \eta^2 \ex{\| \varepsilon_{t+1} \|_2^2} \\
    &\leq \tfrac{4}{(1+\eta \mu)^2} \ex{ \|x_t - x^\star \|_2^2} + \tfrac{4\beta^2}{(1+\eta \mu)^2(4 - (1+\beta)^2)} \ex{ \|x_{t-1} - x^\star \|_2^2 } + \eta^2 \sigma^2, 
\end{align*}
where the last inequality follows from \eqref{eq:igdhb-onestep} and Assumption~\ref{assump:bounded-var}.

\end{proof}

\subsection*{Proof of Lemma \ref{lem:SPPAM-contraction}}

\begin{proof}
\eqref{eq:onestep-acc-stoc-prox} in the main text leads to the following $2 \times 2$ recursion: 
\begin{align}
\label{eq:twobytwo-system}
\begin{bmatrix}
  \ex{\|x_{t+1} - x^\star \|_2^2}  \\
  \ex{\|x_t - x^\star \|_2^2}
\end{bmatrix}
\leq 
\underbrace{
\begin{bmatrix}
  \tfrac{4}{(1+\eta \mu)^2} & \tfrac{4\beta^2}{(1+\eta \mu)^2(4 - (1+\beta)^2)} \\
  1 & 0
\end{bmatrix}
}_{:= A}
\cdot
\begin{bmatrix}
  \ex{\|x_t - x^\star \|_2^2} \\
  \ex{\|x_{t-1} - x^\star \|_2^2} 
\end{bmatrix}
+ 
\begin{bmatrix}
  \eta^2 \sigma^2 \\
  0
\end{bmatrix}.
\end{align}  

Eigenvalues of a $2 \times 2$ matrix $\begin{bmatrix}
  a & b \\ c& d
\end{bmatrix}$ is given by $\frac{(a+d) \pm \sqrt{(a+d)^2 - 4 (ad-bc)}}{2}.$
Thus, eigenvalues of the contraction matrix $A$ is given by 
\begin{align*}
    \sigma_{1, 2} &= \frac{4}{2(1+\eta \mu)^2} \pm \frac{1}{2}\cdot \sqrt{ \left( \frac{4}{(1+\eta \mu)^2} \right)^2 -4 \left( - \frac{4\beta^2}{(1+\eta \mu)^2(4-(1+\beta)^2)} \right) } \\
    &= \frac{2}{(1+\eta \mu)^2} \pm  \sqrt{  \frac{4}{(1+\eta \mu)^4}  + \frac{4\beta^2}{(1+\eta \mu)^2(4-(1+\beta)^2)}  } .
\end{align*}

Note that all terms are positive. Thus, the maximum eigenvalue is determined by 
\begin{align*}
    \sigma_1 = \frac{2}{(1+\eta \mu)^2} + \sqrt{  \frac{4}{(1+\eta \mu)^4}  + \frac{4\beta^2}{(1+\eta \mu)^2(4-(1+\beta)^2)}  } .
\end{align*}
\end{proof}

\subsection*{Proof of Corollary \ref{cor:acc-condition}}

\begin{proof}
We want to see under what condition the following holds:
\begin{align*}
     \frac{1}{1+2\eta \mu} &> \frac{2}{(1+\eta \mu)^2} + \sqrt{  \frac{4}{(1+\eta \mu)^4}  + \frac{4\beta^2}{(1+\eta \mu)^2(4-(1+\beta)^2)}  } \Rightarrow  \\
     \frac{1}{1+2\eta \mu} - \frac{2}{(1+\eta \mu)^2} &> \sqrt{  \frac{4}{(1+\eta \mu)^4}  + \frac{4\beta^2}{(1+\eta \mu)^2(4-(1+\beta)^2)}  } 
\end{align*}
Squaring both sides \footnote{Here, to square both sides and maintain the inequality, we assume $\eta \mu > 1$, which holds by the condition \eqref{eq:init-discount-condition}.} and grouping the same terms, we get 
\begin{align*}
    \left( \frac{1}{1+2\eta \mu} \right)^2 
    - \frac{4}{(1+\eta \mu)^2(1+2\eta \mu)} &>
    \frac{4\beta^2}{(1+\eta \mu)^2(4-(1+\beta)^2)} \Rightarrow \\
    \frac{\eta^2 \mu^2 - 6 \eta \mu - 3}{(1+2\eta \mu)^2} &>
    \frac{4\beta^2}{4-(1+\beta)^2}.
\end{align*}

\end{proof}

\subsection*{Proof of Theorem \ref{thm:lin-conv}}
\begin{proof}
Unrolling the recursion~\eqref{eq:twobytwo-system} for $T$ iterations, we get
\begin{small}
\begin{align}
\begin{bmatrix}
  \ex{\|x_T - x^\star \|_2^2}  \\
  \ex{\|x_{T-1} - x^\star \|_2^2} 
\end{bmatrix}
&\leq 
A^T
\cdot
\begin{bmatrix}
  \|x_0 - x^\star \|_2^2 \\
  \|x_{-1} - x^\star \|_2^2 
\end{bmatrix}
+ 
\left(
\sum_{i=1}^{T-1}A^i\right)
\begin{bmatrix}
  1 \\
  0
\end{bmatrix} 
\eta^2 \sigma^2 \nonumber \\
&= \left( \frac{\sigma_1^T -\sigma_2^T}{\sigma_1 - \sigma_2}A - \sigma_1\sigma_2 \frac{\sigma_1^{T-1} -\sigma_2^{T-1}}{\sigma_1 - \sigma_2}I \right)
\cdot
\begin{bmatrix}
  \|x_0 - x^\star \|_2^2 \\
  \|x_{-1} - x^\star \|_2^2 
\end{bmatrix}
+ 
\left(
\sum_{i=1}^{T-1}A^i\right)
\begin{bmatrix}
  1 \\
  0
\end{bmatrix} 
\eta^2 \sigma^2 \nonumber \\
&\leq  \frac{2 \sigma_1^T}{\sigma_1 - \sigma_2} (A + I) 
\cdot
\begin{bmatrix}
  \|x_0 - x^\star \|_2^2 \\
  \|x_{-1} - x^\star \|_2^2 
\end{bmatrix}
+ 
\left(
\sum_{i=1}^{T-1}A^i\right)
\begin{bmatrix}
  1 \\
  0
\end{bmatrix} 
\eta^2 \sigma^2, \label{eq:unrolled}
\end{align}  
\end{small}
where the last equality is using the formula in \cite{williams1992nth}, and the last inequality is due to 
\begin{small}
\begin{align*}
    \frac{\sigma_1^T - \sigma_2^T}{\sigma_1-\sigma_2} &\leq \frac{|\sigma_1|^T+|\sigma_2|^T}{\sigma_1-\sigma_2} \leq \frac{2|\sigma_1|^T}{\sigma_1-\sigma2} = \frac{2\sigma_1^T}{\sigma_1-\sigma2}, 
\end{align*}
\end{small}
and
\begin{small}
\begin{align*}
    -\sigma_1\sigma_2 \frac{\sigma_1^{T-1} -\sigma_2^{T-1}}{\sigma_1 - \sigma_2} &\leq 
    |\sigma_1 \sigma_2| \frac{|\sigma_1|^{T-1} +|\sigma_2|^{T-1}}{\sigma_1 - \sigma_2} \\
    &= \frac{|\sigma_2| \cdot |\sigma_1|^T - |\sigma_1|\cdot |\sigma_2|^T}{\sigma_1 - \sigma_2} 
    \leq \frac{|\sigma_1|^T  + |\sigma_1|^T}{\sigma_1 - \sigma_2} \leq \frac{2\sigma_1^T}{\sigma_1-\sigma_2},
\end{align*}
\end{small}
under the assumption that $|\sigma_{1,2}| < 1$, which we justified in Theorem~\ref{thm:init-discount-condition}.

Now, focusing on the error term, $\sum_{i=1}^{T-1} A^i$ converge to:
\begin{align*}
    \sum_{i=1}^{T-1} A^i = (I-A)^{-1} (I-A^T) := B (I-A^\top).
\end{align*}
Then, 
\begin{small}
\begin{align} \label{eq:error-term-bound}
    \left(
\sum_{i=1}^{T-1}A^i\right)
\begin{bmatrix}
  1 \\
  0
\end{bmatrix} 
\eta^2 \sigma^2 
&= 
(I-A)^{-1} (I-A^T)
\begin{bmatrix}
  1 \\
  0
\end{bmatrix} 
\eta^2 \sigma^2 \nonumber \\
&= -BA^T \begin{bmatrix}
  1 \\
  0
\end{bmatrix} 
\eta^2 \sigma^2 
+ 
B \begin{bmatrix}
  1 \\
  0
\end{bmatrix} 
\eta^2 \sigma^2 \nonumber \\ 
&\leq B \left( - \frac{\sigma_1^T - \sigma_2^T}{\sigma_1 - \sigma_2} A + \sigma_1 \sigma_2  \frac{\sigma_1^{T-1} - \sigma_2^{T-1}}{\sigma_1 - \sigma_2} I \right) +
B \begin{bmatrix}
  1 \\
  0
\end{bmatrix} 
\eta^2 \sigma^2
\nonumber \\
&\leq B \left( \frac{2\sigma_1^T}{\sigma_1 - \sigma_2} A + \frac{2\sigma_1^T}{\sigma_1 - \sigma_2} I \right) +
B \begin{bmatrix}
  1 \\
  0
\end{bmatrix} 
\eta^2 \sigma^2
\nonumber \\
&= 
\frac{2\sigma_1^T}{\sigma_1 - \sigma_2} B(A+I) 
\begin{bmatrix}
  1 \\
  0
\end{bmatrix} 
\eta^2 \sigma^2 + B \begin{bmatrix}
  1 \\
  0
\end{bmatrix} 
\eta^2 \sigma^2.
\end{align} 
\end{small}
Computing $(I-A)^{-1} := B$ term first, we get
\begin{small}
\begin{align*}
    (I-A)^{-1} &= 
    \left( 1 - \frac{4}{(1+\eta \mu)^2} - \frac{4\beta^2}{(1+\eta \mu)^2 (4-(1+\beta)^2)} \right)^{-1}
    \begin{bmatrix}
  1 & \frac{4\beta^2}{(1+\eta \mu)^2 (4-(1+\beta)^2)} \\
  1 &  1- \frac{4}{(1+\eta \mu)^2}
\end{bmatrix}\\
&:= \frac{1}{p-q} \cdot
  \begin{bmatrix}
  1 & q \\
  1 & p
\end{bmatrix}
\end{align*}
\end{small}
Then,
\begin{small}
\begin{align} \label{eq:coeff}
    B(A+I) \begin{bmatrix}
  1 \\
  0
\end{bmatrix} 
= \frac{1}{p-q} 
\begin{bmatrix}
  2-p+q \\
  2
\end{bmatrix} , \quad \text{and}\quad
B \begin{bmatrix}
  1 \\
  0
\end{bmatrix} 
= \frac{1}{p-q} 
\begin{bmatrix}
  1 \\
  1
\end{bmatrix}.
\end{align}
\end{small}
Combining \eqref{eq:unrolled}, \eqref{eq:error-term-bound}, and \eqref{eq:coeff}, we have
\begin{small}
\begin{align*}
\begin{bmatrix}
  \ex{\|x_T - x^\star \|_2^2}  \\
  \ex{\|x_{T-1} - x^\star \|_2^2} 
\end{bmatrix}
    &\leq  \frac{2 \sigma_1^T}{\sigma_1 - \sigma_2} (A + I) 
\cdot
\begin{bmatrix}
  \|x_0 - x^\star \|_2^2 \\
  \|x_{-1} - x^\star \|_2^2 
\end{bmatrix}
+ 
\left(
\sum_{i=1}^{T-1}A^i\right)
\begin{bmatrix}
  1 \\
  0
\end{bmatrix} 
\eta^2 \sigma^2 \\
&\leq  \frac{2 \sigma_1^T}{\sigma_1 - \sigma_2} 
\left( 
(A + I) 
\cdot
\begin{bmatrix}
  \|x_0 - x^\star \|_2^2 \\
  \|x_{-1} - x^\star \|_2^2 
\end{bmatrix}
+ 
\frac{1}{p-q}
\begin{bmatrix}
  2-p+q \\
  2
\end{bmatrix} \eta^2 \sigma^2 
\right)
+
\frac{1}{p-q} 
\begin{bmatrix}
  1\\1
\end{bmatrix} \eta^2 \sigma^2.
\end{align*}
\end{small}
Since we assume $x_0 = x_{-1}$, we have $\|x_0 - x^\star \|_2^2 = \|x_{-1} - x^\star \|_2^2$. Using this and computing $(A+I)$ explicitly, the top row results in:
\begin{small}
\begin{align*}
    \ex{\|x_T - x^\star \|_2^2} \leq \frac{2 \sigma_1^T}{\sigma_1 - \sigma_2} \left( \left( 2-p+q \right) \cdot  \left( \|x_0 - x^\star\|_2^2 + \frac{\eta^2\sigma^2}{p-q} \right) \right) + \frac{\eta^2 \sigma^2}{p-q}.
\end{align*}
\end{small}
Observe that 
\begin{align*}
     p - q &= 1 - \frac{4}{(1+\eta \mu)^2} - \frac{4\beta^2}{(1+\eta \mu)^2 (4-(1+\beta)^2)}  := 1 - \theta \Rightarrow  \\
     2 - p + q &= 1 + \theta,
\end{align*}
so we can write the above recursion as:
\begin{small}
\begin{align*}
    \ex{\|x_T - x^\star \|_2^2} \leq \frac{2 \sigma_1^T}{\sigma_1 - \sigma_2} \left( \left( 1+\theta \right) \cdot  \left( \|x_0 - x^\star\|_2^2 + \frac{\eta^2\sigma^2}{1-\theta} \right) \right) + \frac{\eta^2 \sigma^2}{1-\theta}.
\end{align*}
\end{small}
Thus, we see that the initial condition is discounted by the factor
\begin{small}
\begin{align*}
    \frac{2 \sigma_1^T}{\sigma_1 - \sigma_2}  =  \tau^{-1} 
    \cdot
    \left( \frac{2}{(1+\eta\mu)^2} + \tau \right)^T
\end{align*}
\end{small}
where 
$\tau = \sqrt{ \tfrac{4}{(1+\eta \mu)^4} + \tfrac{4\beta^2}{(1+\eta \mu)^2(4 - (1+\beta)^2)}},$
up to a region that depends on $O(\eta^2 \sigma^2).$
\end{proof}

\subsection*{Proof of Theorem \ref{thm:init-discount-condition}}

\begin{proof}
We want to analyze the term
\begin{align*}
    \frac{2 \sigma_1^T}{\sigma_1 - \sigma_2}  =
      \frac{\left( \frac{2}{(1+\eta\mu)^2} + \sqrt{ \tfrac{4}{(1+\eta \mu)^4} + \tfrac{4\beta^2}{(1+\eta \mu)^2(4 - (1+\beta)^2)}} \right)^T }{\sqrt{ \tfrac{4}{(1+\eta \mu)^4} + \tfrac{4\beta^2}{(1+\eta \mu)^2(4 - (1+\beta)^2)}}}.
\end{align*}
First, notice that $\frac{2}{(1+\eta\mu)^2} \leq \sqrt{ \tfrac{4}{(1+\eta \mu)^4} + \tfrac{4\beta^2}{(1+\eta \mu)^2(4 - (1+\beta)^2)}}.$ 
Thus, 
\begin{small}
\begin{align*}
    \frac{2 \sigma_1^T}{\sigma_1 - \sigma_2} 
    &\leq
    \frac{\left( 2 \sqrt{ \tfrac{4}{(1+\eta \mu)^4} + \tfrac{4\beta^2}{(1+\eta \mu)^2(4 - (1+\beta)^2)}} \right)^T }{\sqrt{ \tfrac{4}{(1+\eta \mu)^4} + \tfrac{4\beta^2}{(1+\eta \mu)^2(4 - (1+\beta)^2)}}} \\
    &= 2 \cdot \left( 2 \sqrt{ \tfrac{4}{(1+\eta \mu)^4} + \tfrac{4\beta^2}{(1+\eta \mu)^2(4 - (1+\beta)^2)}}  \right)^{T-1} .
\end{align*}
\end{small}
Therefore, if $2 \sqrt{ \tfrac{4}{(1+\eta \mu)^4} + \tfrac{4\beta^2}{(1+\eta \mu)^2(4 - (1+\beta)^2)}} < 1,$ we have exponential discount of the initial conditions. This condition leads to the desired result immediately. Also note that this condition justifies $|\sigma_{1,2}| < 1$, which we required in the proof of Theorem \ref{thm:lin-conv}.
\end{proof}

\section{Derivation of Algorithm \ref{alg:sppam-glm}}

In this section, we present the derivation of the procedure in Algorithm~\ref{alg:sppam-glm}. Note that the following derivation is based on implicit SGD procedure presented in \cite{toulis_statistical_2014}, extended to SPPAM.

We have 
\begin{align*}
  x_t &= x_{t-1} + \eta \nabla f_i(x_t) + \beta(x_{t-1} - x_{t-2}) \\
  &= x_{t-1} + \eta \left( b_{i_t} - h(x_t^\top a_{i_t}) \right) a_{i_t} + \beta(x_{t-1} - x_{t-2}),
\end{align*}
where in the last equality we substituted the gradient for GLM for a (uniformly sampled) single data point $(a_{i_t}, b_{i_t})$, with $h(\cdot)$ being the \textit{mean} function from the main text.

First, multiply both sides by $a_{i_t}$. Then,
\begin{align*}
    x_t^\top a_{i_t} &= x_{t-1}^\top a_{i_t} + \eta \left( b_{i_t} - h(x_t^\top a_{i_t}) \right) a_{i_t} ^\top a_{i_t} + \beta(x_{t-1} - x_{t-2})^\top a_{i_t}.
\end{align*}
Now let $\xi_t := \eta \left( b_{i_t} - h(x_t^\top a_{i_t}) \right) \in \mathbb{R}$.
Then, we have:
\begin{align*}
  x_t^\top a_{i_t} = \xi_t \|a_{i_t}\|_2^2 + (1+\beta) x_{t-1}^\top a_{i_t} - \beta x_{t-2}^\top a_{i_t}.
\end{align*} 
We now apply the transfer function $h(\cdot)$ on both sides to get:
\begin{equation}
\label{eq:derivation}
  h(x_t^\top a_{i_t}) = h\left( \xi_t \|a_{i_t}\|_2^2 + (1+\beta) x_{t-1}^\top a_{i_t} - \beta x_{t-2}^\top a_{i_t} \right).
\end{equation}
But from $\xi_t = \eta \left( b_{i_t} - h(x_t^\top a_{i_t}) \right)$, we can re-arrange to get:
\begin{equation*}
  h(x_t^\top a_{i_t}) = b_{i_t} - \frac{\xi_t}{\eta }.
\end{equation*}
Plugging this back into the left-hand side of \eqref{eq:derivation} and solving for $\xi_t$, we have:
\begin{align*}
  \xi_t &= \eta \left( y - h(\xi_t \|a_{i_t}\|_2^2 + (1+\beta) x_{t-1}^\top a_{i_t} - \beta x_{t-2}^\top a_{i_t}) \right) \\
  x_t &= (1+\beta)x_{t-1} + \xi_t a_{i_t} - \beta x_{t-2} = x_{t-1} + \xi_t a_{i_t} + \beta(x_{t-1} - x_{t-2}),
\end{align*}
arriving at Algorithm~\ref{alg:sppam-glm}. Note that this derivation is assuming a single data point is sampled, but the derivation for mini-batch version is straightforward.


\end{document}